\documentclass[12pt,reqno]{amsart}
\usepackage{amsthm,amsfonts,amssymb,euscript}
\usepackage[all,cmtip]{xy}

\def\ddb#1{\sqrt{-1}\partial\bar{\partial}#1}
\def\dbar#1{\bar{\partial}#1}

\newcommand{\C}{{\mathbb C}}

\newcommand{\R}{{\mathbb R}}
\newcommand{\CP}{{\mathbb{C}P}}

\newcommand{\dr}{\omega}
\newcommand{\al}{\alpha}
\newcommand{\be}{\beta}
\newcommand{\de}{\delta}
\newcommand{\e}{\epsilon}
\newcommand{\ka}{K\"{a}hler}
\newcommand{\MA}{Monge-Amp\'{e}re}
\newcommand{\Sc}{\mathcal{S}}
\newcommand{\Cc}{\mathcal{C}}
\theoremstyle{plain}
  \newtheorem{theorem}{Theorem}[section]
  \newtheorem{proposition}[theorem]{Proposition}
  \newtheorem{lemma}[theorem]{Lemma}
  \newtheorem{corollary}[theorem]{Corollary}

  \newtheorem{remark}[theorem]{Remark}

\theoremstyle{definition}
  \newtheorem{definition}[theorem]{Definition}

\numberwithin{equation}{section}
\begin{document}

\title[Partial $C^{0}$-estimate]{The Partial $C^{0}$-estimate along a general continuity path and applications}
\author{Ke Feng}
\author{Liangming Shen$^1$}

\address{School of Mathematical Sciences, Peking University, Beijing, 100871, P.R. China}
\email{kefeng@math.pku.edu.cn}
\address{School of Mathematical Sciences, University of Science and Technology of China, Hefei, Anhui, 230026, P.R.China}
\email{lmshen@ustc.edu.cn}
\thanks{$^1$ Partially supported by CIRGET Fellowship when the second author visited ISM of UQ\`{A}M}

\begin{abstract}
We establish a new partial $C^{0}$-estimate along a continuity path mixed with conic singularities along a simple normal crossing divisor and a positive
 twisted $(1,1)$-form on Fano manifolds. As an application, this estimate enables us to show the reductivity of the automorphism group of the limit space, which leads to a proof of Yau-Tian-Donaldson Conjecture admitting some types of holomorphic vector fields.
\end{abstract}

\maketitle


\section{Introduction}\label{section 1}

Finding canonical metrics on \ka\ manifolds is the central problem in \ka\ geometry. In 1970s in the celebrated work \cite{Yau}, Yau solved Calabi's conjecture and established the existence of Ricci flat \ka\ metric on \ka\ manifolds with $c_{1}(M)=0.$ Aubin \cite{Au} and Yau also established the
existence of \ka-Einstein metric with negative Ricci curvature on \ka\ manifolds with $c_{1}(M)<0.$ The main idea is to establish a priori estimates for the
solutions to the family of complex \MA\ equations along a continuity path. The remaining problem is the Fano case, i.e.,
$c_{1}(M)>0.$ Unlike the two cases above, Matsushima \cite{Ma} and Futaki \cite{Fu} showed that there are obstructions to the existence of \ka-Einstein metrics
on Fano manifolds thus there are Fano manifolds which do not admit \ka-Einstein metrics.

To solve the \ka-Einstein problem on Fano manifolds, Tian made a crucial progress in \cite{Ti90} which first introduced the partial $C^{0}$-estimate. Let us recall the basic settings of this problem: Let $(M,\dr_{0})$ be a Fano manifold with a \ka\ metric $\dr_{0}\in [2\pi c_{1}(M)],$ which satisfies that $Ric(\dr_{0})=\dr_{0}+\ddb h$ where $h$ a smooth $\dr_0$-psh function on $M.$ Suppose $\dr=\dr_{0}+\ddb\varphi$ is the \ka-Einstein metric on $M$ then $\varphi$
satisfies the following complex \MA\ equation $$(\dr_{0}+\ddb\varphi)^{n}=e^{h-\varphi}\dr_{0}^{n}$$ where $\varphi$ satisfies that $\int_{M}e^{h-\varphi}\dr_{0}^{n}=V.$ To solve this equation, a standard way as \cite{Au,Yau} is to establish the solution to the following continuous family of \MA\ equations
\begin{equation}\label{eq:ma-aubin}
(\dr_{0}+\ddb\varphi_{t})^{n}=e^{h-t\varphi_{t}}\dr_{0}^{n}
\end{equation}
for $t\in[0,1]$ with $\int_{M}e^{h-t\varphi_{t}}\dr_{0}^{n}=V.$
Tian realized that it was impossible to derive a priori $C^{0}$-estimate of \eqref{eq:ma-aubin} directly as \cite{Yau}. Instead, he noted that different embeddings of the Fano manifold into a projective space $\mathbb{CP}^{N_{l}}$ by holomorphic sections $S_{0},S_{1},\cdots,S_{N_{l}}$ of the line bundle $K_{M}^{-l}$ induce a family of \ka\ metrics $\dr_{0}+\ddb\psi$ parameterized by the automorphism group of the projective space, where $\psi=\log\rho-\sup\log\rho$ with
\begin{equation}\label{eq:bergman}
\rho:=\rho_{\dr,l}=\sum\limits_{i=0}^{N_{l}}||S_{i}||_{\dr}^{2}
\end{equation}
and $||\cdot||_{\dr}$ is the Hermitian metric on $K_{M}^{-l}$ induced by $\dr$ with the normalization condition $\int_{M}||S_{i}||_{\dr}^{2}\dr^{n}=1.$ Finally he established the partial $C^{0}$-estimate and showed the reductivity of the automorphism group of the manifold implies the existence of the \ka-Einstein metric in case of Fano surfaces. Later, in \cite{Ti97} Tian extended his idea and proposed $K$-stability and conjectured that this condition implies the existence of \ka-Einstein metrics on Fano manifolds. In \cite{Ti97,Ti12} Tian also pointed out that the partial $C^{0}$-estimate is the crucial step to solve the conjecture. In 2012 Tian \cite{Ti15}, and Chen-Donaldson-Sun \cite{CDS1,CDS2,CDS3} proved the partial $C^{0}$-estimate along the continuity path of conical \ka-Einstein metrics and finally solved this folklore conjecture. Recently in \cite{LTW1,LTW2} Li-Tian-Wang solved the Yau-Tian-Donaldson conjecture in case of singular Fano varieties.

As the most crucial part in the study of \ka-Einstein problem, the partial $C^{0}$-estimate has become an interesting topic in \ka\ geometry. In \cite{Ti12},
Tian introduced a partial $C^{0}$-conjecture in a more general form that there exists a uniform partial $C^{0}$-estimate for a family of Fano manifolds with
uniform positive lower Ricci curvature bound and fixed volume. Besides the partial $C^{0}$-estimates in \cite{Ti15,CDS2,CDS3} along the continuity path of conical \ka-Einstein metrics, Donaldson-Sun \cite{DS} and Tian \cite{Ti13} also considered the partial $C^{0}$-estimates on \ka-Einstein manifolds. Besides those works, there are some works such as \cite{CW,Ji,Sz,TZ} which consider the partial $C^0$-estimates under different settings.

Another important ingredient in the study of \ka-Einstein problem is the conic \ka\ metrics. As a natural generalization of \ka-Einstein metrics, the conical \ka-Einstein metrics were studied in \cite{Ber,GP,JMR,LS,LSh,Ru,TW1,TW2} and played the important role in the solution to the Yau-Tian-Donaldson Conjecture. In fact the conical \ka-Einstein metrics with deforming cone angles give rise to the continuity path which establishes the existence of the smooth \ka-Einstein metric as soon as cone angle attains $2\pi.$

In this paper, we consider a general continuity path $\{\dr_t\}_{t\in[0,T)}$ with $T\in(0,1]$ on the Fano manifold $M$ as following:
\begin{equation}\label{eq:path}
Ric(\dr_t)=t\dr_t+(1-t)(\sum_{r=1}^{m}2\pi b_r[D_r]+b_0\al_0),
\end{equation}
where
\begin{itemize}
\item $\al_0$ is a smooth positive $(1,1)$ form in $c_1(M);$
\item $D_1,\cdots,D_m$ are semi-ample irreducible divisors with simple normal crossings;
\item $b_0,b_1=\frac{p_1}{q},\cdots,b_m=\frac{p_m}{q}$ are positive rational numbers less than $1$ where $p_1,\cdots,p_n,q$ are least integers in the definition of $b_1,\cdots,b_m;$
\item $\sum_{i=r}^{m}b_r[D_r]\in(1-b_0)c_1(M)$
\item $\{\dr_t\}\in c_1(M)$ are family of \ka\ metrics with conic singularities along $D_r.$
\end{itemize}
Actually this path \eqref{eq:path} is a generalization to the path of conical \ka-Einstein metrics in \cite{Ti15,CDS1,CDS2,CDS3} and Aubin's path in \cite{Sz}.
Our main result is the partial $C^0$-estimate along this general continuity path \eqref{eq:path}:
\begin{theorem}\label{thm-main}
For $l=l_i \to\infty,$ there exist constants $c_l >0$ such that for any $x\in M,\;t\in[0,T)$
\begin{equation}\label{eq:partial C0}
\rho_{\dr_t,l}(x)>c_{l}.
\end{equation}
\end{theorem}
By the partial $C^0$-estimate as above, similar to the previous works, as $t\to T$ there exist a family of automorphisms $\sigma_i\in G=SL(N_{l}+1)$ $(M,\cup_{r}D_r,\dr_{t_i})$ which give rise to an element in the stabilizer $G_\infty$ of the Gromov-Hausdorff limit $(M_\infty,D_{\infty}\dr_\infty).$
Then we can show that
\begin{corollary}\label{cor-reductivity}
The Lie algebra $\eta_\infty$ of $G_\infty$ is reductive.
\end{corollary}
Finally, by the partial $C^0$-estimate and the corollary above, we have the following version of Yau-Tian-Donaldson Conjecture:
\begin{corollary}\label{cor-YTD}
Given a Fano manifold $M$ without holomorphic vector fields which induce equivariant actions on $M.$ If $M$ is $K$-stable, there exists a \ka-Einstein metric on $M.$
\end{corollary}
Compared with Datar-Sz\'ekelyhidi's $G$-equivariant case \cite{DSz}, our result focuses on the case that no holomorphic vector fields induce equivariant automorphisms as it is difficult to preserve those divisors in \eqref{eq:path} under equivariant automorphisms.

Let us briefly describe the main ideas of this paper. We mainly follow the steps in \cite{Ti15} combined with some ideas from \cite{Sz}. First we need to establish the geometric limit structure as $t\to T.$ For this target we need to approximate
the conic metric in \eqref{eq:path} by smooth metrics with uniform Ricci lower bound, which was done by the second author in \cite{Sh} based on the techniques in \cite{Ti15}. Next we make use of Cheeger-Colding-Tian theory combined with Carron's technique \cite{Ca} as \cite{Sz} to establish the limit structure, especially the structure of the tangent cone. Then we can modify the smooth convergence part in \cite{Ti15} and then establish the partial $C^0$-estimate. Finally we follow \cite{Ti15} to complete the two corollaries.

\noindent{\bf Acknowledgment.} The authors want to express their sincerely acknowledgement to Professor Gang Tian and Xiaohua Zhu for suggesting this problem and a lot of discussions and encouragements. They also want to thank Chi Li, Zhenlei Zhang and Feng Wang for their beneficial advice on this work. Finally the second author wants to thank ISM of UQ\'AM, McGill University and BICMR for their hospitality during this work.

\section{The structure of the Gromov-Hausdorff limit space}

On \ka\ manifold $M$ associated with a simple normal crossing divisor $D=\sum_{r=1}^{m}D_r,$ if a positive $(1,1)$-current $\dr$ is a smooth metric out of
$D$ and without loss of generality, in the local holomorphic coordinate chart around a point $p$ lying in the intersection of $D_1,\cdots,D_k$ where locally $D_r=\{z_r=0\},$ $\dr$ is asymptotically to the model metric
$$\dr_{cone}:=\sqrt{-1}\left(\sum_{r=1}^{k}\frac{dz_r\wedge d\bar{z}_r}{|z_{r}|^{2(1-\be_r)}}+\sum_{r=k+1}^{n}dz_r\wedge d\bar{z}_r\right),$$ then we call
$\dr$ is a conic \ka\ metric with cone angles $2\pi\be_r$ along $D_r$ for $r=1,\cdots,k.$
To apply the classical Cheeger-Colding-Tian theory \cite{CC,CCT} to the conic metrics, as \cite{Ti15}, we need to approximate the conic metric defined in \eqref{eq:path} by smooth \ka\ metrics with uniform lower Ricci curvature bound. By \eqref{eq:path} it follows that $Ric(\dr_t)\geq t\dr_t,$ then we have the following approximation theorem by the second author in \cite{Sh}:
\begin{theorem}\emph{(Shen \cite{Sh})}\label{thm-shen}
Given a \ka\ manifold $(M,\dr_0)$ with $D=\sum_{r=1}^{m}D_r$ which have simple normal crossings and each component $D_i$ is irreducible and semi-ample, and a
conic \ka\ metric $\dr=\dr_{0}+\ddb\varphi$ with cone angles $2\pi\be_r\ (0<\be_r<1, 1\leq r\leq m)$ along $D_{i}$ and $\varphi\in Psh(M,\dr_0)$ is smooth on $M\setminus D$. If the conic \ka\ metric $\dr$ satisfies that $Ric(\dr)\geq\mu\dr,$ then for any $\de>0$, there exists a smooth \ka\ metric $\dr_{\de}$ in the same \ka\ class to $\dr$ satisfying that $Ric(\dr_\de)\geq\mu\dr_\de$ which converges to $\dr$ in the Gromov-Hausdorff topology on $M$ and in the smooth topology outside $D$ as $\delta$ tends to $0.$
\end{theorem}
By Theorem \ref{thm-shen}, for any sequence $t_i\to T$ and $\dr_i:=\dr_{t_i},$ there exist a sequence of smooth \ka\ metrics $\tilde{\dr}_i$ satisfying that
\begin{enumerate}
\item [(A1)] $[\tilde{\dr}_i]=[\dr_i]\in 2\pi c_1(M);$
\item [(A2)] $Ric(\tilde{\dr}_i)\geq t_i\tilde{\dr}_i;$
\item [(A3)] the Gromov-Hausdorff distance $d_{GH}(\dr_i,\tilde{\dr}_i)\leq 1/i.$
\end{enumerate}
Then by the Gromov compactness theorem, without loss of generality $(M,\tilde{\dr}_i)$ converges to a metric space $(M_\infty,d_\infty)$ in the Gromov-Hausdorff topology. And it follows from (A3) that $(M,\dr_i)$ also converges to $(M_\infty,d_\infty)$ in the Gromov-Hausdorff topology. Similar to Theorem 3.1 of \cite{Ti15} we have the following
geometric description of $(M_\infty,d_\infty):$
\begin{theorem}\label{thm-limit space}
There is a closed subset $\Sc\subset M_\infty$ of Hausdorff codimension at least $2$ such that $M_\infty\setminus\Sc$ is a smooth \ka\ manifold and $d_\infty$ is induced by a twisted \ka-Einstein metric $\dr_\infty$ outside $\Sc$ which satisfies that $Ric(\dr_\infty)=T\dr_\infty+(1-T)b_0\al_0.$ If $T<1$ $\dr_i$ converges to $\dr_\infty$ in $C^\infty$ topology. If $T=1$ the singular set $\Sc$ has codimension at least $4$ and $\dr_\infty$ extends to a smooth \ka-Einstein metric on $M_\infty\setminus\Sc.$
\end{theorem}
\begin{proof}
The proof is almost the same to \cite{Ti15}. We also set $\mathcal{R}$ as the regular part of $M_\infty$ where the tangent cone is $\R^{2n}.$ When $T=1,$ we can show that
$$\int_{M}|Ric(\tilde{\dr}_i)-\tilde{\dr}_i|\tilde{\dr}_i^{n}\leq\int_{M}(Ric(\tilde{\dr}_i)-t_i\tilde{\dr}_i)\tilde{\dr}_i^{n}
+(1-t_i)\int_{M}\tilde{\dr}_i^{n}=2(1-t_i)\int_{M}\tilde{\dr}_i^{n}\to 0,$$
which implies that $(M,\tilde{\dr}_i)$ composite an almost \ka-Einstein sequence defined by Tian-Wang in \cite{TWb}. By \cite{TWb} the singular set $\Sc$ in $M_\infty$ has codimension at least 4 and $d_\infty$ is induced by the smooth \ka-Einstein metric $\dr_\infty$ on $M_\infty\setminus\Sc.$

When $T<1$ away from divisors, similar to \cite{Ti15}, we first show that $\mathcal{R}$ is open. For any $x\in\mathcal{R}$ there are a sequence of points $x_i\in(M,\dr_i)$ which converge to $x.$ Then by the definition of $\mathcal{R}$ that the volume of balls $B_{\bar{r}}(x_i,\dr_i)$ in $(M,\dr_i)$ are uniformly close to the Euclidean volume thus those balls should be away from the divisors along which the metrics are asymptotically conical. In this case apply the proposition \ref{prop-sze8} below which is modified from \cite{Sz} it follows that the Ricci curvature of balls $B_{\frac{\bar{r}}{2}}(x_i,\dr_i)$ is uniformly bounded. Thus it follows from Anderson's $\e$-regularity theorem \cite{An} that there is a uniform curvature bound holds $B_{\frac{\bar{r}}{4}}(x_i,\dr_i)$ in $(M,\dr_i)$ which implies that $\dr_i$ smoothly converges to a twisted \ka-Einstein metric $\dr_\infty$ on $M_\infty\setminus\Sc,$ and moreover it satisfies that $Ric(\dr_\infty)=T\dr_\infty+(1-T)b_0\al_0$ and induces $d_\infty.$ Furthermore the stratification of $\Sc$ with codimension at least 2 are all the same. In fact the proof shows that the limit of divisors $D_r$ lies in $\Sc.$
\end{proof}

Let us recall some basic facts from \cite{CC}. The stratification of $\Sc$ can be described as
$\Sc_0\subset\Sc_1\subset\cdots\subset\Sc_{2n-1},$ where $\Sc_k$ denotes the subset of $\Sc$ where no tangent cone can splits off a factor $\R^{k+1}.$ A tangent cone $\Cc_x$ at $x\in M_\infty$ can be defined as the limit of $(M_\infty,r_{j}^{-1}d_\infty,x)$ for a subsequence $r_j\to 0.$ By Theorem \ref{thm-limit space} when $T<1$ the singular set $\Sc=\Sc_{2n-2}$ and when $T=1$ the singular set $\Sc=\Sc_{2n-4}.$

To give a detailed characterization of the tangent cone, we need the following proposition modified from Proposition 8 in \cite{Sz}, which also has been used in the proof of Theorem \ref{thm-limit space}:
\begin{proposition}\label{prop-sze8}
Fix $T_0,T_1\in(0,1),$ for any $t<T\in(T_0,T_1),$ there exists $R_0,\de>0$ depending on $T_0,T_1$ such that if $I(B_{r_0}(p,\dr_t))>1-\de,$ for $r_0\leq R_0$ then$$|Ric(\dr_t)|\leq 5$$ on $B_{r_0}(p,\dr_t)$ where $I(\Omega_g):=\displaystyle\inf\limits_{B_r(x)\subset\Omega}\frac{Vol B_r(x,g)}{Vol B_r(0,\R^{2n})}.$
\end{proposition}
\begin{proof}
The only different issue from \cite{Sz} is that all points lies on the divisors $D=\sum D_r$ violate the condition on $I(B_{\dr_t}(p,r_0))$ due to the asymptotical behavior of the conic metric. Away from divisors the argument is almost the same to Proposition 8 in \cite{Sz}.
\end{proof}
Now we can show the following structure theorem of the tangent cone, similar to \cite{Ti15} which is modified from \cite{CCT}:
\begin{theorem}\label{thm-tangent cone}
Let $\Cc_x$ be a tangent cone of $M_\infty$ at $x\in\Sc,$ then we have the following:
\begin{enumerate}
\item [(C1)]Each $\Cc_x$ is regular outside a closed subcone $\Sc_x$ of complex codimension at least 1. Such
$\Sc_x$ is the singular set of $\Cc_x.$
\item [(C2)]$\Cc_x=\C^{k}\times\Sc'_x,$ in particular $\Sc_{2k+1}=\Sc_{2k}$ and we denote $o$ as the vertex of $\Cc_x.$
\item [(C3)]There is a natural \ka-Ricci flat metric $g_x$ whose \ka\ form $\dr_x$ is $\ddb\rho_x^2$ on $\Cc_x\setminus\Sc_x,$ where $\rho_x$ denotes the distance function from $o.$ Also $g_x$ is a cone metric.
\item [(C4)]For any $x\in\Sc_{2n-2}$ with $\Cc_x=\C^{n-1}\times\Cc'_x,$ then $\Cc'_x$ is a two-dimensional flat cone of angle $2\pi\overline{\be}.$ Moreover there exist $\overline{\be}_0,\overline{\be}_1\in(0,1)$ such that $\overline{\be}_0\leq\overline{\be}\leq\overline{\be}_1.$
\end{enumerate}
\end{theorem}
\begin{proof}
As \cite{Ti15}, the proof of (C1)-(C3) can be modified from \cite{CCT} directly. For (C4), we still need a modified slice argument in \cite{CCT,Ti15}. Since $(M,\dr_i)$ converge to $(M_\infty,\dr_\infty)$ there are
$r_i\to 0$ and $x_i\in M$ such that $(M,r_i^{-2}\dr_i,x_i)$ converges to the cone $\Cc_x.$ Thus there are
$\e_i\to 0$ and map $(\Phi_i,u_i): B_{5/2}(x_i,r_i^{-2}\dr_i)\longmapsto\C^{n-1}\times\R_+$ satisfying that
$$\max\{Lip(\Phi_i),Lip(u_i)\}\leq c(n),\qquad\int\limits_{|z|<1,|z|\in\C^{n-1}}|V(z)-2\pi\overline{\be}|dz\wedge d\bar{z}\leq\e_i,$$
where $V(z)$ is the volume of $\Sigma_z=\Phi_i^{-1}(z)\cap u_i^{-1}([0,1]$ with respect to $r_i^{-2}\dr_i.$
As denoted by \cite{Ti15} we can assume that $\Phi_i$ is smooth along divisor $D.$

Now write $\e=\e_i, (\Phi,u)=(\Phi_i,u_i).$ By the estimate above, we can find a subset $B_\e$ of
$\{|z|<1\}\subset\C^{n-1}$ with large measure such that for any $z\in B_\e,\Sigma_z$ is transversal to $D=\sum D_r$ with its boundary converging to $\{z\}\times S_{\overline{\be}}^1$ as $i\to\infty,$ where $S_{\overline{\be}}^1$ denotes the unit circle in $\Cc'_x$ and $|V(z)-2\pi\overline{\be}|\leq C\e$ for some uniform constant $C.$ Furthermore, write $\Phi:=(h_1+\sqrt{-1}h_2,\cdots,h_{2n-3}+\sqrt{-1}h_{2n-2}),$ by \cite{CCT} we have the following estimates:
\begin{enumerate}
\item [(i)]$\int_{\Sigma_z}|\langle\nabla h_k,\nabla h_l\rangle-\de_{kl}|=o(1);$
\item [(ii)]$\int_{\Sigma_z}|hess h_k|=o(1);$
\item [(iii)]$\int_{\Sigma_z}|\langle\nabla h_k,\nabla u^2\rangle|=o(1);$
\item [(iv)]$\int_{\Sigma_z}|\nabla\langle\nabla h_k,\nabla u^2\rangle|=o(1).$
\end{enumerate}
Now $K_M^{-1}$ restrict to a line bundle on $\Sigma_z$ with an induced Hermitian metric
$h_z$ by $r_i^{-2}\dr_i$ whose curvature $\Omega$ is equal to
$$Ric(r_i^{-2}\dr_i)=t_i\dr_i+(1-t_i)(\sum_{r=1}^{m}2\pi b_r\iota^*_z[D_r]+b_0\al_0).$$
Now denote $\Sigma_{z,\de}:=\Sigma_{z}\setminus\cup_s B_\de(p_s)$ where $p_s$ denote points where divisors $D_r$ transversally intersect with $\Sigma_z$ and $\de$ is small enough. Suppose each $D_r$ has $m_r$ intersections (accounting multiplicities) with $\Sigma_z$ and denote $\be_{r,i}=1-(1-t_i)b_r$ as the corresponding cone angles of $\dr_i$ along $D_r.$  Set $K$ as the Gaussian curvature on $\Sigma_z,$ by the estimates of $\Phi$ above we can see that the second fundamental form $\Pi_{\Sigma_z}$ in $(M,r_i^{-2}\dr_i)$ is small, it follows that
$$\int_{\Sigma_{z,\de}}K=\int_{\Sigma_{z,\de}}Ric(r_i^{-2}\dr_i)+o(1)=(1-t_i)b_0\int_{\Sigma_{z,\de}}r_i^2\al_0+o(1).$$
By Gauss-Bonnet formula, it follows that
$$2\pi(\chi(\Sigma_{z}-\sum_r m_r))=2\pi\chi(\Sigma_{z,\de})=\int_{\Sigma_{z,\de}}K+\int_{\partial\Sigma_{z,\de}}H-2\pi\sum_r m_r\be_r,i$$ where $H$ denotes the geodesic curvature on the boundary. Let $i\to\infty$ and $\de\to 0,$ it follows that
$$1-\overline{\be}\geq\sum_r m_r(1-\be_r)+o(1)$$ where $\be_r:=1-(1-T)b_r.$ Thus if there are at least one $m_r>0$ we have the upper bound $\overline{\be}\leq\overline{\be}_1$ for some $\overline{\be}_1<1.$

On the other hand, if all $m_r=0,$ i.e., $\Sigma_z$ has no intersection with all divisors, we can adapt the argument in
Proposition 11 of \cite{Sz}. Suppose there are sequence of points $\{y_i\}\subset\Sc_{2n-2}=\Sc\subset M_\infty$ such that they have tangent cones $\Cc_{y_i}$ with cone angles $2\pi\gamma_i$ converging to $2\pi.$
Note that as $\Sc$ is closed we may assume that $y_i\to y\in\Sc.$ then there are sequence of points $x_i\in(M,\dr_i)$ and $r_i\to 0$ such that
$$Vol B_{r_i^{-2}\dr_i}(x_i,1)=\gamma_i Vol_{\R^{2n}}B(0,1)+o(1)$$ and the limit of a subsequence of those scaled balls around $x_i$ is the tangent cone $\Cc_y$ where $y\in\Sc_{2n-2}.$
Thus as $\gamma_i\to 1,$ by Proposition \ref{prop-sze8}, $r_i^{-2}\dr_i$ has bounded Ricci curvature on the half ball $B_{1/2}(x_i,r_i^{-2}\dr_i).$ However by Cheeger-Colding-Tian theory \cite{CCT} it follows that the singular set $\Sc$ of the limit of this sequence $(B_{1/2}(x_i,r_i^{-2}\dr_i),r_i^{-2}\dr_i)$ has at least codimension 4, which leads to a contradiction to $y\in\Sc_{2n-2}.$ Thus $\gamma_i$ cannot be arbitrarily close to 1 and we establish a uniform upper bound $\overline{\be}\leq\overline{\be}_1$ for all possibilities. The lower bound $\overline{\be}_0$ can be derived from standard Bishop-Gromov volume comparison theorem.
\end{proof}

\section{Smooth Convergence}

As \cite{Ti15}, we need to show that $\dr_i$ in the last section smoothly converges to $\dr_\infty$ outside a closed subset of codimension at least 2, which is crucial for the partial $C^0$-estimate. When $T<1$ actually we have shown that the limit of divisors lie in the singular set $\Sc$ as $\dr_i$ smoothly converges to $\dr_\infty$ outside $\Sc.$ However when $T=1$ the singular set $\Sc$ has codimension at least 4 but each $\dr_i$ is smooth outside the divisors which have codimension 2, thus the conclusion is not clear. In this section we will follow the strategy in \cite{Ti15} to show the conclusion of smooth convergence.

For our continuity path \eqref{eq:path}, we need to establish a new construction of Hermitian metrics on $K_M^{-1}.$ According to the assumptions of \eqref{eq:path}, first as $\al_0\in c_1(M),$ by Yau's solution
to Calabi's Conjecture \cite{Yau}, there exists a smooth \ka\ metric $\dr_{\al_0}\in c_1(M)$ such that
$Ric(\dr_{\al_0})=\al_0.$ Now set $\tilde{H}_{\al_0}:=\dr_{\al_0}^{n}$ as a Hermitian metric on $K_M^{-1}$
then it follows that $R(\tilde{H}_{\al_0})=Ric(\dr_{\al_0})=\al_0.$ Next recall that $\dr_{i}=\dr_{t_i}$ in
\eqref{eq:path} is a conic metric with cone angles $2\pi\be_{r,i}$ along $D_r$ where $\be_{r,i}=1-(1-t_i)b_r.$
For simplicity we skip the index $i,$ then the Hermitian metric $\tilde{H}_{\dr}:=\dr^{n}$ defined on $K_{M}^{-1}$ has the order $\prod\limits_r |S_r|^{-2(1-\be_r)}$ along $D=\sum\limits_r D_r,$ where $S_r$ is the defining section of $D_r.$ By the conditions that $b_r=\frac{p_r}{q}$ and $\sum\limits_r b_r[D_r]\in(1-b_0)c_1(M),$ we construct the following holomorphic section
$$S:=S_1^{p_1}\otimes S_2^{p_2}\otimes\cdots\otimes S_r^{p_r}\in H^0(M,q(1-b_0)K_M^{-1}).$$
Then we can construct a Hermitian metric
\begin{equation}\label{eq:hermitian}
H_\dr:=\tilde{H}_{\al_0}(S,S)^{-c_1}\tilde{H}_{\dr}(S,S)^{c_2}\tilde{H}_{\dr}
\end{equation}
on $K_M^{-1}$ with the constants $c_1,c_2$ such that $H_\dr$ is nonsingular and its curvature $R(H_\dr)=\dr.$
For those targets, as $S\in H^0(M,q(1-b_0)K_M^{-1}),$ we can show that
$$\tilde{H}_{\al_0}(S,S)\sim\prod_r|S_r|^{2p_r},\quad\tilde{H}_{\dr}(S,S)\sim\displaystyle\frac{\prod_r|S_r|^{2p_r}}{\prod_r |S_r|^{2(1-\be_r)q(1-b_0)}}.$$
Thus to make $H_\dr$ nonsingular for each $r$ it suffices to satisfy
$$-c_1 p_r+c_2(p_r-q(1-b_0)(1-\be_r))-(1-\be_r)=0.$$
On the other hand, for the second requirement, we need
\begin{align*}
R(H_\dr)&=-c_1 q(1-b_0)R(\tilde{H}_{\al_0})+(1+c_2 q(1-b_0))R(\tilde{H}_{\dr})-2\pi(c_2-c_1)\sum_r p_r[D_r]\\
&=(1+c_2 q(1-b_0))Ric(\dr)-2\pi(c_2-c_1)\sum_r p_r[D_r]-c_1 q(1-b_0)\al_0=\dr.
\end{align*}
To satisfy the equations above, it suffices to have
\begin{equation}\label{cond}
\left\{
   \begin{array}{ll}
    \displaystyle q(c_2-c_1)&=(1-t)(1+c_2 q(1-b_0))\\
    \displaystyle c_1 q(1-b_0)&=b_0(1-t)(1+c_2 q(1-b_0)),
    \end{array}
 \right.
\end{equation}
Thus we have $$c_1=b_0 c_2=\frac{b_0(1-t)}{tq(1-b_0)}, c_2=\frac{1-t}{tq(1-b_0)},$$ which give us the required $H_\dr,$ which is called the associated Hermitian metric of $\dr.$ Moreover for any $\sigma\in H^0(M,K_M^{-l})$ it follows that $H_\dr(\sigma,\sigma)$ is bounded.

The remaining steps are almost the same to \cite{Ti15} so we only sketch the main steps in the
following. First, by Theorem \ref{thm-shen}, similar to the argument in the proof of Lemma 3.3 in
\cite{Ti15}, we have uniform Sobolev Inequalities with respect to all $\dr_i.$ As Lemma 4.1 in
\cite{Ti15} we also have equations for $\sigma\in H^0(M,K_M^{-l}):$
\begin{align*}
\Delta_i||\sigma||_i^2&=||\nabla\sigma||_i^2-nl||\sigma||_i^2,\\
\Delta_i||\nabla\sigma||_i^2&=||\nabla^2\sigma||_i^2-((n+2)l-t_i)||\nabla\sigma||_i^2,
\end{align*}
where $||\cdot||_i$ denotes the Hermitian metric on $K_M^{-l}$ induced by $H_i=H_{\dr_i},$
$\nabla$ denotes the covariant derivative of $H_i,$ and $\Delta_i$ denotes the Laplacian of $\dr_i.$
By the uniform Sobolev Inequalities and the equations above we have the following uniform estimates which follows from Moser's iteration:
\begin{lemma}\label{lem-T4.2}\emph{(\cite{Ti15} Corollary 4.2)}
If $\sigma_i$ is a sequence in $H^0(M,K_M^{-l})$ satisfying that $\int_M||\sigma||_i^2\dr_i^{n}=1,$ then
\begin{equation}\label{eq:section uniform}
\sup_{M}(||\sigma_i||_i+l^{-1/2}||\nabla\sigma_i||_i)\leq Cl^{n/2}.
\end{equation}
\end{lemma}
It follows from Lemma \ref{lem-T4.2} that $||\sigma_i||_i$ are uniformly continuous. Thus by taking
a subsequence if necessary we may assume $||\sigma_i||_i$ converge
to a Lipschtz function $F_\infty$ as $i\to\infty$ and $\int_{M_\infty}F_\infty^2\dr_\infty^n=1.$ We can see
that $F_\infty$ is not identical to 0. We only need to show that $\dr_i$ converge to $\dr_\infty$ away
from $F_\infty^{-1}(0)\cup\Sc$ and $F_\infty$ is equal to the square norm of a holomorphic section on $M_\infty.$

If $F_\infty(x)\neq 0$ for some $x\in M_\infty,$ we can argue that for $x_i\to x$ where $x_i\in(M,\dr_i)$ the small balls
$B_r(x_i,\dr_i)$ are away from divisors. In this case the volumes of those small balls are close to Euclidean balls.
Then similar to the proof of Theorem \ref{thm-limit space}, by \cite{An} we can
show the uniform curvature estimate and consequently the smooth convergence follows away from $F_\infty^{-1}(0)\cup\Sc.$ Furthermore we can show that $M_\infty\setminus F_\infty^{-1}(0)$
is dense. In fact if there exists an open set $U\subset F_\infty^{-1}(0)$ then as $||\sigma_i||_i$
is bounded from above, it follows that
$$\lim\limits_{i\to\infty}\int_M\log\left(\frac{1}{i}+||\sigma_i||_i^2\right)\dr_i^n=-\infty.$$
However by direct computations $$\ddb\log\left(\frac{1}{i}+||\sigma_i||_i^2\right)\geq-l\dr_i.$$
thus by standard Moser iteration and the uniform Sobolev Inequality in Lemma 3.3 of \cite{Ti15},
it follows that
$$\sup_M\log\left(\frac{1}{i}+||\sigma_i||_i^2\right)\leq C\left(1+
\int_M\log\left(\frac{1}{i}+||\sigma_i||_i^2\right)\dr_i^n\right),$$
which tends to $-\infty$ as $i\to\infty.$ However this contradicts the fact that $\int_M||\sigma||_i^2\dr_i^{n}=1,$ thus $M_\infty\setminus F_\infty^{-1}(0)$ is dense.

As $\dr_i\to\dr_\infty$ smoothly away from $F_\infty^{-1}(0)\cup\Sc,$ $\sigma_i$ converges
to a holomorphic section $\sigma_\infty$ on $M_\infty$ with $F_\infty=||\sigma_\infty||_\infty.$ As $\sigma_\infty$ is uniformly bounded and holomorphic on dense set, it could be extended to a holomorphic section on $M_\infty\setminus\Sc.$ As $||\sigma_i||_i=0$ on $D$ the limit of $D$ must lies in $D_\infty=\{F_\infty=0\}.$ On the other hand, if $D_\infty$ does not coincide with the limit of $D,$ there exist $x\in D_\infty$ and $r>0$ such that $B_{2r}(x,d_\infty)\cap D_\infty$
is disjoint from the limit of $D.$ Then for sufficiently large $i,$ $B_{r}(x,\dr_i)$ is disjoint from $D$ thus lies in the smooth part of $(M,\dr_i).$ By Proposition \ref{prop-sze8} the Ricci curvature on $B_{r/2}(x,\dr_i)$ is uniformly bounded and it follows from Cheeger-Colding-Tian theory \cite{CCT} that $\Sc\cap B_{r}(x,d_\infty)$ is of complex codimension at least 2 and near a generic point $y\in B_{r}(x,d_\infty)$ $\sigma_\infty$ is holomorphic and defines $D_\infty.$ By the smooth convergence of $(M,\dr_i)$ to $(M_\infty,d_\infty)$ near $y$ it follows that $\sigma_i(y)=0$ which contradicts with the assumption that $y$ is not in the limit of $D.$ Thus $D_\infty$ must coincide with the limit of $D.$

If $T=1$ as $\Sc$ is of complex codimension at least 2 then $D_\infty=\{\sigma_\infty=0\},$ which
is a divisor of $K_{M_\infty}^{-l}.$ To summarize, we have the following theorem as Theorem 4.3 in \cite{Ti15}:
\begin{theorem}\label{thm-smooth}
$(M,\dr_i)$ converge to $(M_\infty,\dr_\infty)$ in the $C^\infty$-topology outside a closed subset
$\overline{\Sc}\cup D_\infty$ where $\overline{S}$ is of codimension at least $4,$ and $D$ converges to
$D_\infty$ in the Gromov-Hausdorff topology. If $T>1,$ $\Sc=\overline{\Sc}\cup D_\infty.$ If $T=1,$
$\Sc=\overline{\Sc}$ and $D_\infty$ is a divisor of $K_{M_\infty}^{-l}.$
\end{theorem}

\section{Partial $C^0$-Estimate}

In this section we will follow \cite{Ti15} to prove Theorem \ref{thm-main}, i.e., establish the partial
$C^0$-estimate. In fact by last two chapters, it suffices to show the partial $C^0$-estimates
\eqref{eq:bergman} with respect to the sequence $\dr_i=\dr_{t_i}$ for $t_i\to T.$

By Theorem \ref{thm-smooth}
we have shown that $\dr_i$ smoothly converge to $\dr_\infty$ outside $\Sc$ when $T<1$ or outside
$\Sc\cup D_\infty$ when $T=1.$ Meanwhile for any section $\sigma_i\in H^0(M,K_M^{-l})$ which has the same support with $D$ and satisfies $\int_M H_i(\sigma_i,\sigma_i)\dr_i^n=1$ where $H_i=H_{\dr_i}$ was defined in the last chapter, then $\sigma_i$ converges to $\sigma_\infty\in H^0(M_\infty,K_{M_\infty}^{-l})$ as $\dr_i$
converges smoothly to $\dr_\infty.$ Moreover, $\sigma_\infty$ also has the same support with the divisor $D_\infty$ which is the limit of $D.$

In particular, choosing $\sigma_i$ as $S=S_1^{p_1}\otimes S_2^{p_2}\otimes\cdots\otimes S_r^{p_r}\in H^0(M,q(1-b_0)K_M^{-1})$ normalized by $H_i$ thus the limit $\sigma_\infty\in H^0(M_\infty,q(1-b_0)K_{M_\infty}^{-1}).$ Similar to the last section, we can define a Hermitian metric $H_\infty$ on $K_{M_\infty}^{-l}$ over $M_\infty\setminus\Sc$ by $$H_\infty=\tilde{H}_{\al_0}(\sigma_\infty,\sigma_\infty)^{-c_1}\tilde{H}_{\infty}
(\sigma_\infty,\sigma_\infty)^{c_2}\tilde{H}_{\infty},$$
where $$c_1=b_0 c_2=\frac{b_0(1-T)}{Tq(1-b_0)}, c_2=\frac{1-T}{Tq(1-b_0)},$$ and
$\tilde{H}_{\al_0},\tilde{H}_{\infty}$ are defined by $\dr_{\al_0}^{n},\dr_\infty^n$ on $K_{M_\infty}^{-1}.$
Note that $\dr_{\al_0}$ can be defined on $M_\infty\setminus\Sc$ by the smooth convergence.

Similar to Lemma 5.2 and 5.3 in \cite{Ti15}, we can easily derive the following two lemmas:
\begin{lemma}\label{lem-T5.2}
The Hermitian metrics $H_i$ converge to $H_\infty$ on $M_\infty\setminus\Sc$ in the $C^\infty$-topology.
Moreover we have
$$H_\infty(\sigma_\infty,\sigma_\infty)<\infty\;\emph{and}\;
\int_{M_\infty}H_\infty(\sigma_\infty,\sigma_\infty)\dr_\infty^n=1.$$
\end{lemma}
\begin{lemma}\label{lem-T5.3}
For any $l>0,$ if $\{\tau_i\}$ is any sequence of $H^0(M,lK_M^{-1})$ satisfying
$$\int_M H_i(\tau_i,\tau_i)\dr_i^{n}=1,$$ then a subsequence $\tau_i$ converges to a section
$\tau_\infty\in H^0(M_\infty,lK_{M_\infty}^{-1}).$
\end{lemma}
Now we can begin the proof of Theorem \ref{thm-main}. The main steps are quite similar to \cite{Ti15} so here we only sketch the main steps of the proof. First we note that by the definition \eqref{eq:bergman}, the fact that $||\sigma||_i$ and their covariant derivatives are uniformly bounded and the finite covering argument, we only need to show that for any $x\in M_\infty,$ there is an $l=l_x$ and a sequence $x_i\in M$ such that $x_i\to x$ and
\begin{equation}\label{eq:bergman-local}
\inf\limits_i\rho_{\dr_i,l}(x_i)>0.
\end{equation}
Next we need the following lemma on H\"ormander's $L^2$-estimate for the $\dbar$-operator on $(M,\dr_i),$ which
follows from the $L^2$-estimate with respect to the approximating metrics constructed in Theorem
\ref{thm-shen}:
\begin{lemma}\label{lem-L2}
For any $l>0,$ given a $(0,1)$ form $\zeta$ with values in $K_M^{-l}$ with $\dbar\zeta=0,$ there
is a smooth section $\vartheta$ of $K_M^{-l}$ such that $\dbar\vartheta=\zeta$ and moreover,
$$\int_{M}||\vartheta||_i^2\dr_i^n\leq\frac{1}{l+t_i}\int_{M}||\zeta||_i^2\dr_i^n,$$
where $||\cdot||_i$ denotes the norm induced by $\dr_i$ as before.
\end{lemma}
Next, recall that in Theorem \ref{thm-tangent cone} we established the existence and basic properties of tangent cones $\Cc_x$ for $x\in(M_\infty,\dr_\infty).$ Then by taking a subsequence if necessary, for $r_j\to 0,$ we have a tangent cone $\Cc_x$ at $x$ which is the Gromov-Hausdorff
limit of $(M_\infty,r_j^{-2}\dr_\infty,x)$ satisfying:
\begin{enumerate}
\item [(T1)]Each $\Cc_x$ is regular outside a closed subcone $\Sc_x$ of complex codimension at least 1, where $\Sc_x$ is the singular set of $\Cc_x.$
\item [(T2)]There is a nature \ka\ Ricci-flat metric $g_x$ on $\Cc_x\setminus\Sc_x$ which is also a cone metric, with the \ka\ form $\ddb\rho_x^2$ on the regular part of $\Cc_x$ where
    $\rho_x$ denotes the distance function from the vertex $o$ of $\Cc_x.$
\end{enumerate}
Denote $L_x=\Cc_x\times\C$ as the trivial line bundle over $\Cc_x$ equipped with the Hermitian metric $e^{-\rho_x^2}|\cdot|^2$ thus the curvature of this metric is just $\dr_x.$ For any $\e>0$
put $$V(x;\e):=\{y\in\Cc_x|y\in B_{\e^{-1}}(0,g_x)\setminus\overline{B_\e(0,g_x)},d(y,\Sc_x)>\e\}.$$ Choose the sequence $r_j\to 0$ such that $r_j^{-2}$ are integers and $(M_\infty,r_j^{-2}\dr_\infty,x)$ converges to $(\Cc_x,g_x,o).$ By \cite{CCT} for any $\e,\de>0$ there exists a $j_0$ such that for any $j\geq j_0$ there is a diffeomorphism $\phi:V(x;\e/4)\mapsto M_\infty\setminus\Sc$ satisfying:
\begin{enumerate}
\item [(1)]$d(x,\phi(V(x;\e)))<10\e r$ and $\phi(V(x;\e))\subset B_{(1+\e^{-1})r}(x)$ where $r=r_j.$
\item [(2)]Denote $g_\infty$ as the \ka\ metric with the \ka\ form $\dr_\infty,$ then
\begin{equation}\label{eq:app-cone}
||r^{-2}\phi^*g_\infty-g_x||_{C^6(V(x;\e/2)}\leq\de,
\end{equation}
 with respect to $g_x.$
\end{enumerate}
From the settings above, we have the following lemma:
\begin{lemma}\label{lem-T5.7}\emph{(\cite{Ti15} Lemma 5.7)}
Given any $\e>0$ and small $\de>0$ there are a sufficiently large $l=r^{-2},$ a diffeomorphism
$\phi:V(x;\e/4)\mapsto M_\infty\setminus\Sc$ with properties (1) and (2) above and an isomorphism
$\psi$ from $\Cc_x\times\C$ onto $K_{M_\infty}^{-l}$ over $V(x;\e)$ commuting with $\phi$ satisfying
\begin{equation}\label{eq:app-cone-bundle}
||\psi(1)||^2=e^{-\rho_x^{2}}\quad\emph{and}\quad ||\nabla\psi||_{C^4(V(x;\e))}\leq\de,
\end{equation}
where $||\cdot||$ denotes the induced norm on $K_{M_\infty}^{-l}$ by $\dr_\infty$ and
$\nabla$ denotes the covariant derivative with respect to $||\cdot||$ and $e^{-\rho_x^2}|\cdot|^2.$
\end{lemma}
\begin{proof}
We sketch the main steps of the proof and suggest \cite{Ti15} for the details. By the approximation construction on the tangent cone above, we can cover $V(x;\e')$ by finite many
geodesic balls $B_{s_\al}(y_\al)(1\leq\al\leq N)$ satisfying that $\overline{B_{2s_\al}(y_\al)}$
is strongly convex and contained in $\Cc_x\setminus\Sc_x,$ $B_{s_\al/2}(y_\al)$ are mutually disjoint, and $s_\al\geq\nu_x d(y_\al,\Sc_x)$ for some constant $\nu_x.$

For $l=r^{-2}$ and $\phi$ above we can first construct bundle morphism $\tilde{\psi}_\al$ over balls $B_{2s_\al}(y_\al).$ Let $\gamma_y\subset B_{2s_\al}(y_\al)$ be the unique minimizing geodesic connecting $y_\al$ and $y\in B_{2s_\al}(y_\al).$ At $y_\al$ define $\tilde{\psi}_\al(1)\in L|_{\phi(y_\al)}$ such that $$||\tilde{\psi}_\al(1)||^2=e^{-\rho_x^2(y_\al)}$$ where $L=K_{M_\infty}^{-l}.$
Then for $y\in U_\al:=B_{2s_\al}(y_\al),$ set $a(y)$ and $\tau(\phi(y))$ as the parallel transportation of 1 and $\tilde{\psi}_\al(1)$ along $\gamma_y$ and $\phi(\gamma_y)$
with respect to the norms $e^{-\rho_x^2}|\cdot|^2$ and $||\cdot||,$ then we can define
$$\tilde{\psi}_\al(a(y))=\tau(\phi(y)).$$
By this construction obviously the first condition in \eqref{eq:app-cone-bundle} holds for $\tilde{\psi}_\al$ over $U_\al.$ For the derivative estimate in the second one, note that
$a:U_\al\mapsto U_\al\times\C$ and $\tau:U_\al\mapsto\phi^*L|_{U_\al}$ satisfy that
$\tilde{\psi}_\al(a)=\tau,$ we have $$\nabla\tau=\nabla\tilde{\psi}_\al(a)+\tilde{\psi}_\al(\nabla a),$$ where the covariant derivative is taken with respect to the metrics defined above. By
the definition of $\tilde{\psi}_\al$ it follows that $\nabla\tilde{\psi}_\al(y_\al)=0.$ At $y,$
take the covariant derivative along $\gamma_y$ it follows that
$$\nabla_T\nabla_X\tau=\nabla_T(\nabla_X\tilde{\psi}_\al(a))+\tilde{\psi}_\al(\nabla_T\nabla_X a),$$ where $T$ is the unit tangent of $\gamma_y$ andy $X$ is a vector field along $\gamma_y$
with $[T,X]=0,$ and note that $\nabla_T\tilde{\psi}_\al=0$ by the definition. Similarly, switch
$X$ and $T$ it follows that $$\nabla_X\nabla_T\tau=\tilde{\psi}_\al(\nabla_X\nabla_T a).$$
Take the difference of these two formulas it follows from the curvature formula that
$$l\phi^*\dr_\infty(T,X)\tilde{\psi}_\al(a)=\nabla_T(\nabla_X\tilde{\psi}_\al(a))+\dr_x(T,X)a.$$
Note that as $l\to\infty,$ $l\phi^*\dr_\infty$ converges to $\dr_x,$ it follows that $\nabla_T(\nabla_X\tilde{\psi}_\al(a))$ is quite small when $l$ is sufficiently large. As
$\nabla_X\tilde{\psi}_\al=0$ at $y_\al,$ $||\nabla\tilde{\psi}_\al||_{C^0(U_\al)}$ can be made
sufficiently small. Higher derivative estimates could be concluded by inductions.

Next we need to modify each $\tilde{\psi}_\al.$ By the construction above, for any $\al,\gamma,$
we have the transition function $$\theta_{\al\gamma}=\tilde{\psi}_\al^{-1}\circ\tilde{\psi}_\al:U_\al\cap U_\gamma\mapsto S^1$$
by the first conclusion for $\tilde{\psi}_\al$ in \eqref{eq:almost holo sec}. Those transition functions compose a closed cycle $\{\theta_{\al\gamma}\}.$ From the derivative estimate of $\tilde{\psi}_\al$ each $\theta_{\al\gamma}$ is close to a constant. Thus we can multiply each $\tilde{\psi}_\al$ by a suitable unit function such that each $\theta_{\al\gamma}$ is a unit constant and all the derivative estimates for $\tilde{\psi}_\al$ are still quite small. Then the modified cycles $\theta_{\al\gamma}$ give rise to a flat bundle $F$ which essentially induces an
isomorphism $\xi:F\mapsto K_{M_\infty}^{-l}$ over an neighborhood of $\overline{V(x;\e')}$ satisfying all the estimates in \eqref{eq:almost holo sec}. The isomorphism $\xi$ induces an isomorphism $\xi^{k}:F^k\mapsto K_{M_\infty}^{-kl}$ for each $k.$ Furthermore, choose $E_\e\subset\partial B_1(o,g_x)\setminus\Sc_x$ with
$$d_{\partial B_1(o,g_x)}(E_\e,\Sc_x\cap\partial B_1(o,g_x))\geq\e^2$$
such that the topology of $E_\e$ depends only on $\e$ and $\Sc.$ Then choose $\e'$ and set
$$U(x;\e',\e)=\{y\in\Cc_x|\sqrt{\e'}<d(x,y)<\e^{-1},\frac{y}{d(x,y)}\in E_\e\}$$ such that
$U(x;\e',\e)\subset V(x,\e').$

The flat bundle $F|_{U(x;\e',\e)}$ is given by a representation
$$\rho:\pi_1(U(x;\e',\e))=\pi_1(E_\e)\mapsto H_1(E_\e,\mathbb{Z})\mapsto S^1.$$
Note that $H_1(E_\e,\mathbb{Z})$ is the sum of an abelian group of finite rank $m$ and a finite group of order $\nu$ which depend only on $\e$ and $\Sc.$ Thus we can choose $k$ such that
$F^k$ is essentially trivial on the scale of $\de,$ i.e., the corresponding transition functions
are in a $\de'$-neighborhood of the identity in $S^1$ where $\de'$ depends on and is much smaller than $\de.$

Replace the original $l$ by $kl$ and set $\e'$ such that
$$k^{-1/2}V(x;\e):=\{y\in\Cc_x|\e<\sqrt{k}d(x,y)<\e^{-1},\sqrt{k}d(y,\Sc_x)>\e\}\subset U(x;\e',\e).$$ Correspondingly, redefine $\phi$ as the original $\phi$ composed with the scaling
$y\mapsto k^{-1/2}y$ on $\Cc_x.$ Since $(M_\infty,kr_j^{-2}\dr_\infty,x)$ also converge to the cone $(\Cc_x,g_x,x)$ all the properties and constructions before also follow. Moreover the new flat bundle $F$ which is the $k$-th power of the original one will have transition functions $\de'$-close to the identity for $\de'\ll\de.$ Thus we can modify $\tilde{\psi}_\al$ slightly and finally we complete all the constructions satisfying the Lemma.
\end{proof}

Now we can begin to show \eqref{eq:bergman-local}, and consequently Theorem \ref{thm-main}.
The main idea is same to \cite{Ti90,Ti15}. First, we need to construct an approximated holomorphic
section $\tilde{\tau}$ on $M_\infty$ then perturb it into a holomorphic section $\tau$ on $(M,\dr_i)$ by the smooth convergence result and the $L^2$-estimate for $\dbar$-operators. Finally
by the derivative estimate we can conclude that $\tau(x)\neq 0.$

Let $\e,\de>0$ sufficiently small and be determined later. Fix $l=r^{-2}$ where $r=r_j$ for sufficiently large $j$ such that the conditions of Lemma \ref{lem-T5.7} holds. Then by Lemma
\ref{lem-T5.7}, choose $\phi,\psi$ and thus there is a section $\tau=\psi(1)$ of $K_{M_\infty}^{-l}$ on $\phi(V(x;\e))$ satisfying
\begin{equation}\label{eq:almost holo sec}
||\tau||^2=e^{-\rho_x^2}\quad\textmd{and}\quad||\dbar\tau||\leq C\de
\end{equation}
for some uniform constant $C.$

Next we need the following technical lemma:
\begin{lemma}\label{lem-T5.8}\emph{(\cite{Ti15} Lemma 5.8)}
For any $\bar{\e}>0,$ there is a smooth function $\gamma_{\bar{\e}}$ on $\Cc_x$ satisfying:
\begin{enumerate}
\item [(1)]$\gamma_{\bar{\e}}(y)=1$ if $d(y,\Sc_x)\geq\bar{\e}$ where $d$ is the distance of $(\Cc_x,g_x).$
\item [(2)]$0\leq\gamma_{\bar{\e}}\leq 1$ and $\gamma_{\bar{\e}}(y)=0$ in an neighborhood of $\Sc_x.$
\item [(3)]$|\nabla\gamma_{\bar{\e}}|\leq C$ for some constant $C=C(\bar{\e})$ and
$$\int_{B_{\bar{\e}^{-1}}(o,g_x)}|\nabla\gamma_{\bar{\e}}|^2\dr_x^n\leq\bar{\e}.$$
\end{enumerate}
\end{lemma}
\begin{proof}
See \cite{Ti15} for the details of the proof. Note that $\Sc_x$
is the union of $\Sc_x^0$ and $\bar{\Sc}_x$ where $\Sc_x^0$ is the part where all points
$y$ have a tangent cone of the form $\C^{n-1}\times\Cc'_y$ with the standard cone metric, and
$\bar{\Sc}_x$ is a subcone of complex codimension at least 2. In
the simplest case $\Sc_x=\C^{n-1},$ we can set a cutoff function $\eta$ satisfying that
$0\leq\eta\leq 1,|\eta'|\leq 1,$ and
$$\eta(t)=0\;\textmd{for}\;t>\log(-\log\bar{\de}^3)\quad\textmd{and}\quad\eta(t)=1\;\textmd{for}\;t<\log(-\log\bar{\de}).$$
Then we can define $$\gamma_{\bar{\e}}(y)=\eta\left(\log\left(-\log\left(\frac{\rho(y)}{\bar{\e}}\right)\right)\right),$$
and we have $$\nabla\gamma_{\bar{\e}}(y)=\frac{\eta'\nabla\rho(y)}{\rho(y)\log\left(\frac{\rho(y)}{\bar{\e}}\right)}.$$
By the standard computations, it follows that
$$\int_{B_{\bar{\e}^{-1}}(o,g_x)}|\nabla\gamma_{\bar{\e}}|^2\dr_x^n\leq
\frac{a_{n-1}}{\bar{\e}^{2n-2}}\int_{\bar{\de}^3}^{\bar{\de}}\frac{dr}{r(-\log r)^2}
\leq\frac{a_{n-1}}{\bar{\e}^{2n-2}(-\log\bar{\de})}.$$
Choose suitable constant $\bar{\de}$ we can make the integration is less than $\bar{\e}$ and
moreover it follows that $|\nabla\gamma_{\bar{\e}}|\leq C(\bar{\e}).$

In general, recall that $(\Cc_x,g_x,o)$ is the limit of $(M_i,r_i^{-2}\dr_i,x_i).$ Thus there are $\de_i$ tending to 0 and diffeomorphisms
$$\tilde{\phi}_i:V(x;\de_i)\mapsto M\setminus T_{\de_i}(D)\quad T_{\de_i}(D)=\{z|d_i(z,D)\leq\de_i\},$$ where $d_i$ denotes the distance with respect to $r_i^{-2}\dr_i$ satisfying
$$||r_i^{-2}\tilde{\phi}_i^*\dr_i-\dr_x||_{C^6(V(x;\de_i)}\leq\de_i.$$ Suppose $l_i=r_i^{-2}$ are integers.

Now for $y\in\Sc_x^0,$ by the assumption there are integers $k_j=s_j^{-2}$ such that $(\Cc_x,k_jg_x,y)$ converge to $(\Cc_y,g_{\bar{\be}},o)$ where $\Cc_y=\C^{n-1}\times\Cc'_y$ with
the flat conic metric $g_{\bar{\be}}$ as above. Thus there are diffeomorphisms
$$\vartheta_j:V(y;j^{-1})\subset\Cc_y\mapsto\Cc_x\setminus\Sc_x$$
satisfying
$$||s_j^{-2}\vartheta_j^*\dr_x-\dr_{\bar{\be}}||_{C^6(V(y;j^{-1})}\leq\frac{1}{j}.$$
Now it follows from the constructions above that for any $\de>0$ there exist $i_\de,j_\de$ such
that for any $i\geq i_\de,j\geq j_\de,$ we have $\tilde{\phi}_i\cdot\vartheta_j:V(y;j^{-1})\mapsto
M\setminus T_{\de_i}(D)$ satisfying
\begin{equation}\label{eq:T-A5}
||k_jl_i\vartheta_j^*\tilde{\phi}_i^*\dr_i-\dr_{\bar{\be}}||_{C^6(V(y;j^{-1})}\leq\de.
\end{equation}

Consider $\Cc_y\times\C$ as a bundle over $\Cc_y$ with the norm $e^{-|z'|^2-|z_n|^{2\bar{\be}}},$
take $f_0=\al_0,f_1=\al_1 z_1,\cdots,f_n=\al_n z_n,$ where $\al_0,\cdots,\al_n>0$ are chosen such that
\begin{equation}\label{eq:T-A6}
\int_{\C^{n-1}\times\Cc'_y}|f_k|^2e^{-|z'|^2-|z_n|^{2\bar{\be}}}\dr_{\bar{\be}}^n=1.
\end{equation}
Apply Lemma \ref{lem-T5.7} to each $f_k,$ then for sufficiently large $i,j$ there exist isomorphisms $\psi_{i,j}:\Cc_x\times\C\mapsto K_M^{-k_jl_i}$ over $V(y;j^{-1})$ satisfying
\begin{equation}\label{eq:T-A7}
||\psi_{i,j}(f_k)||^2=|f_k|^2e^{-|z'|^2-|z_n|^{2\bar{\be}}}\quad\textmd{and}
\quad||\nabla\psi_{i,j}||_{C^4(V(y;j^{-1})}\leq\de.
\end{equation}

We will see the partial $C^0$-estimate follows from Lemma \ref{lem-T5.8} later. Until now we have proved this lemma in case that $\Sc_x$ is a simple cone as above. Thus by the argument of the partial $C^0$-estimate in the following, we can show that there exist holomorphic sections $S_{i,j}^k$ of $K_M^{-k_jl_i}$ over $M$ such that
\begin{equation}\label{eq:T-A8}
\sup_{V(y;j^{-1})\cap B_{10}(o,g_{\bar{\be}})}|\psi^*_{i,j}S_{i,j}^k-f_k|\leq\frac{\e}{2},
\end{equation}
where $\e$ tends to 0 as $\de$ tends to 0. And moreover by the Moser iteration in the proof of Lemma \ref{lem-T4.2} in the same region we have
\begin{equation}\label{eq:T-A9}
||\nabla S_{i,j}^k||_i\leq C
\end{equation}
with respect to the Hermitian norm associated to $\dr_i$ in Section 3. By \eqref{eq:T-A8}
combined with the fact $f_0=\al_0>0,$ it follows that there exists $c>0$ depending only on $\al_0$ such that
\begin{equation}\label{eq:T-A10}
||S_{i,j}^0||_i\geq c\quad\emph{on}\;\tilde{\phi}_i(\vartheta_j(B_{10}(o,g_{\bar{\be}}))).
\end{equation}
Thus we can define a holomorphic map $F_{i,j}:\tilde{\phi}_i(\vartheta_j(B_{10}(o,g_{\bar{\be}})))\mapsto\C^n$ by
\begin{equation}\label{eq:T-A11}
F_{i,j}=\left(\frac{S_{i,j}^1(x)}{S_{i,j}^0(x)},\frac{S_{i,j}^2(x)}{S_{i,j}^0(x)},
\cdots,\frac{S_{i,j}^n(x)}{S_{i,j}^0(x)}\right).
\end{equation}
Then $F_{i,j}\cdot\tilde{\phi}_i\cdot\vartheta_j$ smoothly converge to $(f_1/f_0,\cdots,f_n/f_0)$ outside the singular set $\{z_n=0\}$ as $j,i\to\infty.$
Thus for sufficiently large $i,j$ we have
\begin{equation}\label{eq:T-A12}
|F_{i,j}\cdot\tilde{\phi}_i\cdot\vartheta_j(z)-(f_1/f_0,\cdots,f_n/f_0)(z)|\leq\e
\end{equation}
for any $z\in U_j:=\{(z',z_n)\in B_{10}(o,g_{\bar{\be}})||z_n|^{\bar{\be}}>j^{-1}\}\subset V(y;j^{-1}).$ As $i,j$ are sufficiently large we may assume $B_{8s_jr_i}(x_i,\dr_i)\subset \tilde{\phi}_i\cdot\vartheta_j(B_{10}(o,g_{\bar{\be}}))$ and moreover $F_{i,j}$ is a holomorphic map from $B_{8s_jr_i}(x_i,\dr_i)$ onto its image containing $B_{8-2\e}(o,g_{\bar{\be}})$ for sufficiently small $\e.$ By \eqref{eq:T-A9} it follows that
\begin{equation}\label{eq:T-A13}
\sup_{B_{8s_jr_i}(x_i,\dr_i)}|dF_{i,j}|_{\dr_i}\leq C(s_jr_i)^{-2},
\end{equation}
which implies that
\begin{equation}\label{eq:T-A14}
F_{i,j}^*\dr_0\leq C(s_jr_i)^{-2}\dr_i,
\end{equation}
where $\dr_0$ is the Euclidean metric on $\C^n.$

Next we need to show that for sufficiently large $j,$ $F_{i,j}(D\cap B_{7s_jr_i}(x_i,\dr_i))$ converge to a local divisor $D^j\subset\C^n$ as $i\to\infty$
where $D=\sum_r D_r.$
First we need to bound the volume of $F_{i,j}(D\cap B_{7s_jr_i}(x_i,\dr_i)).$ Recall that
$(\Cc_x,s_j^{-2}g_x,y)$ converge to the standard cone $\C^{n-1}\times\Cc'_y$ with the metric
$g_{\bar{\be}},$ for sufficiently large $i,j$ $F_{i,j}$ maps $D\cap B_{8s_jr_i}(x_i,\dr_i)$
into a tubular neighborhood $T_{8,\e}=\{(z',z_n)||z'|<8,|z_n|<\e\}.$ By the slicing argument in \cite{CCT} (or Theorem \ref{thm-tangent cone}) for each $z'$ with $|z'|<7.5,$ the complex line segment $\{(z',z_n)||z_n|<6\}$ intersects with $F_{i,j}(D_r\cap B_{7s_jr_i}(x_i,\dr_i))$ at $m_r(z')$ points (counted with multiplicity), where $m_r(z')$
satisfies that $1-\bar{\be}\geq\sum_r m_r(z')(1-\be_r)$ as Theorem \ref{thm-tangent cone}.
Moreover, for any such $z'$ there is $m>0$ such that $\sum_r m_r(z')\leq m.$

Let $\tilde{\eta}:\R\mapsto\R$ be a cut-off function satisfying $\tilde{\eta}(t)=1$ for $t\leq 56,\tilde{\eta}(t)=0$ for $t>60,|\tilde{\eta}'(t)|\leq 1,$ and $|\tilde{\eta}''(t)|\leq 2.$ Then we have
\begin{align*}
&\int_{F_{i,j}(D\cap B_{8s_jr_i}(x_i,\dr_i))}\tilde{\eta}(|z'|^2)\dr_0^{n-1}\\
=&\int_{F_{i,j}(D\cap B_{8s_jr_i}(x_i,\dr_i))}\tilde{\eta}(|z'|^2)(\sqrt{-1}dz'\wedge d\bar{z'}+\ddb|z_n|^2)^{n-1}\\
=&\int_{F_{i,j}(D\cap B_{8s_jr_i}(x_i,\dr_i))}(\tilde{\eta}+|z_n|^2(\tilde{\eta}'+(n-1)|z'|^2\tilde{\eta}'')
(\sqrt{-1}dz'\wedge d\bar{z'})^{n-1}.
\end{align*}
It follows that
\begin{equation}\label{eq:T-A15}
\int_{F_{i,j}(D\cap B_{7.4s_jr_i}(x_i,\dr_i))}\dr_0^{n-1}\leq 200nm.
\end{equation}

By the argument in section 3, we can show that the limit of $D$ coincides with $\Sc_x$ modulo a subset of Hausdorff codimension at least 4 under the Gromov-Hausdorff convergence of $(M,r_i^{-2}\dr_i,x_i)$ to $(\Cc_x,\dr_x,o)$ and thus $F_{\infty,j}(\Sc_x\cap B_{7s_j}(y,g_x))$ coincides with $D^j.$

By the monotonicity of the subvariety $F_{i,j}(D)$ combined with \eqref{eq:T-A14}, as $i,j$ are sufficiently large it follows
that
$$1\leq\int_{F_{i,j}(D)\cap B_4(o,\dr_0))}\dr_0^{n-1}\leq
\int_{D\cap B_{6s_jr_i}(x_i,\dr_i)}F^*_{i,j}\dr_0^{n-1}\leq
C\int_{D\cap B_{6s_jr_i}(x_i,\dr_i)}(s_jr_i)^{2-2n}\dr_i^{n-1}.$$
Then as $i\to\infty$ it follows that
\begin{equation}\label{eq:T-A16}
\frac{s_j^{2n-2}}{C}\leq s_j^{2n-2}\int_{\Sc_x\cap B_{6s_j}(y,\dr_x)}F^*_{\infty,j}\dr_0^{n-1}\leq\mathcal{H}_{2n-2}(\Sc_x\cap B_{6s_j}(y,\dr_x)),
\end{equation}
where $\mathcal{H}_{2n-2}$ denotes the $(2n-2)$-dimensional Hausdorff measure with respect to $g_x.$

To summarize, we have the following lemma:
\begin{lemma}\label{lem-T-A1}
For any $\e>0$ small there is a $j_\e$ such that for any $j\geq j_\e$ the Lipschitz map
$F_{\infty,j}$ maps $B_{7s_j}(y,g_x)$ into $B_{7+\e}(o,g_{\bar{\be}})$ satisfying that
\begin{enumerate}
\item Its image contains $B_{7-\e}(o,g_{\bar{\be}}).$
\item $F_{\infty,j}\cap B_{7s_j}(y,g_x)$ is a local divisor $D^j\in T_{8,\e}.$
\item For any $\de>0$ there is a $\nu=\nu(\de)$ such that $F_{\infty,j}^{-1}(T_{6,\nu})\subset T_\de(\Sc_x)\cap B_{7s_j}(y,g_x).$
\end{enumerate}
\end{lemma}
\begin{proof}
We only need to show (3). If  not true, then $F_{\infty,j}^{-1}(D^j\cap B_{6.5}(o,g_{\bar{\be}}))$ has at least two distinct components: one lies in $\Sc_x$ and the other does not. Thus for sufficiently large $i,j$ the preimage
$F_{i,j}^{-1}(F_{i,j}(D^j)\cap B_{6.5}(o,g_{\bar{\be}}))$ has at least two components. On the other hand for sufficiently large $i,j$ restricted on $B_{10s_jr_i}(x_i,\dr_i)\setminus T_\de(D),$ $F_{i,j}$ is biholomorphic onto its image. By \eqref{eq:T-A12} and \eqref{eq:T-A13} $F_{i,j}$ is very close to a coordinate map, i.e., almost separating points
on $B_{8}(o,g_{\bar{\be}}).$ By \eqref{eq:T-A10} $B_{10s_jr_i}(x_i,\dr_i)$ lies in some $\C^{N'_{i,j}}$ thus $F_{i,j}$ is one-to-one on $B_{7s_jr_i}(x_i,\dr_i)$ which leads to a contradiction.
\end{proof}

Next, for sufficiently large $i,j$ there are uniformly bounded functions $\varphi_{i,j}$ on
$B_{8s_jr_i}(x_i,\dr_i)$ satisfying
\begin{equation}\label{eq:T-A17}
(s_jr_i)^{-2}\dr_i=\ddb\varphi_{i,j}.
\end{equation}
The reason is that by the construction $S^0_{i,j}\in K_M^{-k_jl_i}$ is perturbed from the constant $\al_0,$ thus $||S^0_{i,j}||_i$ is close to a uniform constant for sufficiently large $i,j$ and note that $-\ddb\log||S^0_{i,j}||^2_i=k_jl_i\dr_i=(s_jr_i)^{-2}\dr_i.$ Moreover by this observation the volume of $D\cap B_{7s_jr_i}(x_i,\dr_i)$ with respect to
$(s_jr_i)^{-2}\dr_i$ is uniformly bounded. It follows from \eqref{eq:T-A15}, \eqref{eq:T-A17} and basic fact in pluripotential theory that
$$\int_{D\cap B_{7s_jr_i}(x_i,\dr_i)}\dr_i^{n-1}=\int_{D\cap B_{7s_jr_i}(x_i,\dr_i)}(s_jr_i)^{2n-2}(\ddb\varphi_{i,j})^{n-1}\leq C(s_jr_i)^{2n-2}.$$
As $i\to\infty$ we have
\begin{equation}\label{eq:T-A18}
\mathcal{H}_{2n-2}(\Sc_x\cap B_{7s_j}(y,\dr_x))\leq C s_j^{2n-2}.
\end{equation}
Then by a covering argument it follows that for any $R>0$ there is a constant $C_R$ such that
\begin{equation}\label{eq:T-A19}
\mathcal{H}_{2n-2}(\Sc_x\cap B_R(o,\dr_x))\leq C_R.
\end{equation}
We also need the following key lemma:
\begin{lemma}\label{lem-T-A2}
All the notations follows from above and assume that
\begin{enumerate}
\item $\xi:\R\mapsto[0,1]$ is a smooth function with $\xi(t)=1$ for any $t\geq 8\e,$ and
\item $f$ is a holomorphic function on $F_{\infty,j}(B_{7s_j}(y,\dr_x))$ such that $|f(z',z_n)|\geq|z_n|$ whenever $|z_n|\geq 8\e.$
\end{enumerate}
Then there is a uniform constant $C$ such that
$$s_{j}^{2-2n}\int_{B_{6s_j}(y,\dr_x)}|\nabla(h\cdot F_{\infty,j})|^2_{\dr_x}\dr_x^n\leq C\int_{F_{\infty,j}(B_{7s_j}(y,\dr_x))}\sqrt{-1}\partial h\wedge\dbar h\wedge(dz'\wedge d\bar{z}')^{n-1},$$
where $h(z',z_n)=\xi(|f|^2(z',z_n)).$
\end{lemma}
\begin{proof}
It suffices to prove this inequality for all $F_{i,j}$ and then let $i\to\infty.$ Let $\tilde{\eta}:\R\mapsto\R$ be a cutoff function such that $\tilde{\eta}(t)=1$ for $t\leq 40,\tilde{\eta}(t)=0$ for $t>46,|\tilde{\eta}'|\leq 1$ and $|\tilde{\eta}''|\leq 2;$ then
we have $\ddb\tilde{\eta}(|z'|^2)\leq 200n dz'\wedge d\bar{z}'.$

By assumptions (1) and (2) we can see that $\tilde{\eta}(|z'|^2)|dh|^2$ vanishes near the boundary of $F_{i,j}(B_{7s_jr_i}(x_i,\dr_i)).$ It is easy to check $\partial h\wedge\partial\dbar h=0.$ It follows from those facts, \eqref{eq:T-A17} and integration by parts that
\begin{align*}
(s_jr_i)^{2-2n}&\int_{B_{7s_jr_i}(x_i,\dr_i)}\tilde{\eta}(|z'|^2)|\nabla(h\cdot F_{i,j})|^2_{\dr_i}\dr_i^n\\
=n&\int_{F_{i,j}(B_{7s_jr_i}(x_i,\dr_i))}\tilde{\eta}(|z'|^2)|\sqrt{-1}\partial h\wedge\dbar h\wedge(\ddb(\varphi_{i,j}\cdot F_{i,j}^{-1}))^{n-1}\\
\leq C&\int_{F_{i,j}(B_{7s_jr_i}(x_i,\dr_i))}\sqrt{-1}\partial h\wedge\dbar h
\wedge(dz'\wedge d\bar{z}')^{n-1}.
\end{align*}
Then the lemma follows by letting $i\to\infty.$
\end{proof}

To complete the whole construction of the cutoff function in the lemma \ref{lem-T5.8}, let $\bar{\e}$ be given. Fix a small $\e_0>0.$ Since
$\bar{\Sc}_x$  has complex codimension at least 2, we can find a finite cover of $\bar{\Sc}_x\cap B_{\bar{\e}^{-1}}(x,g_x)$ by balls $B_{r_a}(y_a,g_x)\; (a=1,\cdots,l)$ satisfying:
\begin{enumerate}
\item [(i)] $y_a\in\bar{\Sc}_x$ and $2r_a\leq\e_0.$
\item [(ii)] $B_{r_a/2}(y_a,g_x)$ are mutually disjoint.
\item [(iii)]$\sum_a r_a^{2n-3}\leq 1.$
\item [(iv)] The number of overlapping balls $B_{2r_a}(y_a,g_x)$ is uniformly bounded.
\end{enumerate}
Denote $\bar{\eta}$ as a cutoff function $\R\mapsto\R$ satisfying $0\leq\bar{\eta}\leq 1,|\bar{\eta}'|\leq 2,$ and moreover $\bar{\eta}(t)=1$ for $t>1.6$ and $\bar{\eta}(t)=0$ for $t\leq 1.1.$ Set $\chi=\prod_a\chi_a$ where $\chi_a(y)=\bar{\eta}\left(\frac{d(y,y_a)}{r_a}\right)$ if $y\in B_{2r_a}(y_a,g_x)$ and $\chi_a(y)=1$ otherwise. Then $\chi$ vanishes on the closure of $B=\cup_a B_{r_a}(y_a,g_x)$ which contains $\bar{\Sc}_x\cap B_{\bar{\e}^{-1}}(x,g_x)$. Furthermore $\chi$ satisfies
\begin{equation}\label{eq:T-A20}
\int_{\Cc_x}|\nabla\chi|^2\dr_x^n\leq C\sum_a\int_{B_{2r_a}(y_a,g_x)}|\nabla\chi_a|^2\dr_x^n
\leq C\sum_a r_a^{2n-2}\leq C\frac{\e_0}{2}\sum_a r_a^{2n-3}\leq C\e_0.
\end{equation}

There is a finite cover of $\Sc_x\cap B_{\bar{\e}^{-1}}(x,g_x)\setminus B$ by balls $B_{6s_b}(y_b,g_x)$ for which Lemma \ref{lem-T-A1} holds $(b=1,\cdots,N).$ We can also
assume that the number of overlapping balls $B_{6s_b}(y_b,g_x)$ is bounded by a uniform
constant $K.$ Choose smooth functions $\zeta_b$ associated to the cover $\{B_{6s_b}(y_b,g_x)\}$
satisfying
\begin{enumerate}
\item $0\leq\zeta_b\leq 1,\;|\nabla\zeta_b|\leq s_b^{-1}.$
\item $\textmd{supp}(\zeta_b)$ is contained in $\{B_{6s_b}(y_b,g_x)\}.$
\item $\sum_b\zeta_b\equiv 1$ near $\Sc_x\cap B_{\bar{\e}^{-1}}(x,g_x)\setminus B.$
\end{enumerate}
Then $\{\zeta_b\},1-\sum_b\zeta_b$ form a partition of unit for the cover $\{B_{6s_b}(y_b,g_x)\}$ and $B_{\bar{\e}^{-1}}(x,g_x).$

As the proof in the simplest case in the beginning, we denote by $\eta$ a cutoff function
$\R\mapsto\R$ satisfying $0\leq\eta\leq 1,\;|\eta'(t)|\leq 1,$ and
$$\eta(t)=0\;\textmd{for}\;t>\log(-\log\bar{\de}^3)\quad\textmd{and}\quad\eta(t)=1\;
\textmd{for}\;t<\log(-\log\bar{\de}).$$
For each $b$ let $F_b$ be the map from $B_{7s_b}(y_b,g_x)$ into $B_{7+\e}(o,g_{\bar{\be}_b})$
and $D_b\subset B_{7+\e}(o,g_{\bar{\be}_b})$ be the divisor given by Lemma \ref{lem-T-A1}.
Let $\nu$ be in (3) of Lemma \ref{lem-T-A1} for $\de=\bar{\e}.$ It is easy to see that $\nu$
can be chosen independent of $b.$ Let $10\e<\nu$ and $f_b$ be a local defining function of $D_b$ satisfying (2) in Lemma \ref{lem-T-A2}. We define a function $\gamma_{\bar{\e},b}$ on
$B_{7s_b}(y_b,g_x)$ as follows: If $|f_b|(F_b(y))\geq\bar{\e}/3,$ put $\gamma_{\bar{\e},b}(y)=1$ and if $|f_b|(F_b(y))\leq\bar{\e},$ put
\begin{equation}\label{eq:T-A21}
\gamma_{\bar{\e},b}(y)=\eta\left(\log\left(-\log\left(\frac{|f_b|(F_b(y))}
{\bar{\e}}\right)\right)\right).
\end{equation}
Choose $\bar{\de}$ sufficiently small we can deduce from the beginning of
proof that
\begin{equation}\label{eq:T-A22}
\int_{B_{6s_b}(y_b,g_x)}|\nabla\gamma_{\bar{\e},b}|^2\dr_x^n\leq\e_0 s_b^{2n-2}.
\end{equation}
Moreover if $\bar{\e}\bar{\de}<\nu$ then by (3) of Lemma \ref{lem-T-A1} $\gamma_{\bar{\e},b}(y)=1$ if $d(y,\Sc_x)\geq\bar{\e}.$

Now combine all the constructions above and define
\begin{equation}\label{eq:T-A23}
\gamma_{\bar{\e}}(y)=\chi(y)(1-\sum_b\zeta_b(y)+\sum_b\zeta_b(y)
\gamma_{\bar{\e},b}(y)).
\end{equation}
Then $\gamma_{\bar{\e}}(y)=1$ whenever $d(y,\Sc_x)\geq\bar{\e}$ and vanishes
in a neighborhood of $\Sc_x.$ It follows from \eqref{eq:T-A23} and \eqref{eq:T-A20}
that
$$\int_{B_{\bar{\e}^{-1}}(x,g_x)}|\gamma_{\bar{\e}}|^2\dr_x^n\leq C
\left(\e_0+K\sum_b\int_{B_{6s_b}(y_b,g_x)}|\nabla(\zeta_b
(1-\gamma_{\bar{\e},b}))|^2\dr_x^n\right).$$
Assume that $\e\leq\e_0,$ by \eqref{eq:T-A22} and \eqref{eq:T-A16} we have
$$\int_{B_{6s_b}(y_b,g_x)}|\nabla(\zeta_b
(1-\gamma_{\bar{\e},b}))|^2\dr_x^n\leq 4\e_0 s_b^{2n-2}\leq 4\e_0\mathcal{H}_{2n-2}(\Sc_x\cap B_{6s_j}(y_b,g_x)).$$
Combine the estimates above and \eqref{eq:T-A19} it follows that
$$\int_{B_{\bar{\e}^{-1}}(x,g_x)}|\gamma_{\bar{\e}}|^2\dr_x^n\leq
C(1+4K^2\mathcal{H}_{2n-2}(\Sc_x\cap B_{\bar{\e}^{-1}}(o,g_x)))\e_0
\leq C(1+4K^2 C_{\bar{\e}^{-1}})\e_0.$$
The proof of Lemma \ref{lem-T5.8} is complete.
\end{proof}

Now we continue the proof of the partial $C^0$-estimate \eqref{eq:bergman-local}. Define a
cut-off function $\eta$ satisfying
$$\eta(t)=1\;\textmd{for}\;t\leq 1,\quad\eta(t)=0\;\textmd{for}\;t\geq 2,\quad\textmd{and}\;
|\eta'(t)|\leq 1.$$
Let $\de_0>0$ be determined later. Choose $\bar{\e}$ such that $\gamma_{\bar{\e}}=1$ on $V(x;\de).$ Then we choose $\e$ such that $\de_0>4\e$ and $V(x;\e)$ contains the support of $\gamma_{\bar{e}}.$ Now for any $y\in V(x;\e),$ define
\begin{equation}\label{eq:almost holo sec1}
\hat{\tau}(\phi(y))=\eta(2\e(\rho_x(y)+\rho_x(y)^{-1}))\gamma_{\bar{\e}}(y)\tau(\phi(y)).
\end{equation}
It follows from the constructions above that
\begin{equation}\label{eq:almost holo sec2}
\hat{\tau}=\tau\quad\textmd{on}\quad\phi(V(x;\de_0)).
\end{equation}
Moreover it follows from \eqref{eq:almost holo sec} and the fact  $(M_\infty,r_j^{-2}\dr_\infty,x)$ converges to $(\Cc_x,\dr_x,o)$ that
\begin{equation}\label{eq:almost holo sec3}
\int_{M_\infty}||\dbar\hat{\tau}||^2_\infty\dr_\infty^n\leq\nu r^{2n-2},
\end{equation}
where $r=r_j$ and $\nu=\nu(\de,\e)$ which could be sufficiently small as long as $\de,\e,\bar{\e}$ are small.
Moreover we also have
\begin{equation}\label{eq:almost holo sec4}
\int_{M_\infty}||\hat{\tau}||^2_\infty\dr_\infty^n\leq
r^{2n}\int_{\phi(V(x;\e))}e^{-\rho_x^2}\dr_\infty^n<C_0 r^{2n}.
\end{equation}
Now we can see that $\hat{\tau}$ is supported outside the singular set $\Sc$ of $M_\infty.$ Meanwhile we need to modify $\hat{\tau}$ to be supported outside $D_\infty.$ If $T<1$
we know that $D_\infty\subset\Sc$ then we just set $\tilde{\tau}=\hat{\tau}.$ If $T=1$
we could put $\rho=||\sigma_\infty||_\infty$ which was constructed in the last section. Similarly
let $\bar{\eta}$ be a cutoff function satisfying $0\leq\bar{\eta}\leq 1,|\bar{\eta}'|\leq 1$
and $$\bar{\eta}(t)=0\;\textmd{for}\;t>\log(-\log\hat{\e}^2)\quad\textmd{and}\quad\bar{\eta}(t)=1\;
\textmd{for}\;t<\log(-\log\hat{\e}).$$
Define $$\tilde{\tau}(z)=\bar{\eta}(\log(-\log\rho(z)))\hat{\tau}(z),$$ then $\tilde{\tau}$ supports away from $\Sc\cup D_\infty$ and coincides with $\hat{\tau}$ wherever $\rho\geq\hat{\e}.$ When $\hat{\e}$ is small enough, it follows from \eqref{eq:almost holo sec3} that
\begin{equation}\label{eq:almost holo sec5}
\int_{M_\infty}||\dbar\tilde{\tau}||^2_\infty\dr_\infty^n\leq 2\nu r^{2n-2}.
\end{equation}
Now set $U(x;\e)$ to be $\phi(V(x;\e))$ if $T<1$ and $\phi(V(x;\e))\setminus\{z|d_\infty(z,D_\infty)\leq\e\}$ if $T=1.$ If $\hat{\e}$ is sufficiently small $\tilde{\tau}=\tau$ on $U(x;\de_0)$ and the support of $\tilde{\tau}$ is contained in $U(x;\e)$ if $\e$ is small enough.

Recall that $(M\setminus D,\dr_i)$ converge to $(M_\infty\setminus\Sc\cup D_\infty,\dr_\infty)$
and the Hermitian metric $H_i$ defined on $K_{M}^{-l}$ converge to $H_\infty$ on $M_\infty\setminus\Sc\cup D_\infty$ in the $C^{\infty}$-topology. Thus for $\de_i\to 0$ we could define diffeomorphisms
$$\tilde{\phi}_i:M_\infty\setminus T_i(\Sc\cup D_\infty)\mapsto M\setminus T_i(D)$$
associated with smooth isomorphisms $$F_i:K_{M_\infty}^{-l}\mapsto K_{M}^{-l}$$ over $M_\infty\setminus T_i(\Sc\cup D_\infty),$ where $T_i(\Sc\cup D_\infty),T_i(D)$ are
$\de_i$-tubular neighborhood of $\Sc\cup D_\infty$ and $D$ with respect to $\dr_i,\dr_\infty.$
Moreover, we have
\begin{enumerate}
\item [($C_1$)] $\tilde{\phi}_i(M_\infty\setminus T_i(\Sc\cup D_\infty))\subset M\setminus T_i(D).$
\item [($C_2$)] $\pi_i\circ F_i=\tilde{\phi}_i\circ\pi_\infty$ where $\pi_i,\pi_\infty$ are
corresponding projections.
\item [($C_3$)] $||\tilde{\phi}_i^*\dr_i-\dr_\infty||_C^2(M_\infty\setminus T_i(\Sc\cup D_\infty))\leq\de_i.$
\item [($C_4$)] $||F_i^*H_i-H_\infty||_C^4(M_\infty\setminus T_i(\Sc\cup D_\infty))\leq\de_i.$
\end{enumerate}
We set large enough $i$ such that $U(x;\e)\subset M_\infty\setminus T_i(\Sc\cup D_\infty).$ Set
$\tilde{\tau}_i=F_i(\tilde{\tau})$ then it follows that $\tilde{\tau}_i=F_i(\tau)$ on $\tilde{\phi}_i(U(x;\de_0))$ and moreover it follows from \eqref{eq:almost holo sec5} that
\begin{equation}\label{eq:almost holo sec6}
\int_{M}||\dbar\tilde{\tau}_i||^2_i\dr_i^n\leq 3\nu r^{2n-2}.
\end{equation}
By the $L^2$-estimate in Lemma \ref{lem-L2}, there exists a section $v_i$ of $K_M^{-l}$ such that
$\dbar v_i=\dbar\tilde{\tau}_i$ and
$$\int_{M}||\dbar v_i||^2_i\dr_i^n\leq\frac{1}{l}\int_{M}||\dbar\tilde{\tau}_i||^2_i\dr_i^n\leq 3\nu r^{2n}.$$
Take $\sigma_i=\tilde{\tau}_i-v_i,$ which is a holomorphic section of $K_M^{-l}.$ Then it follows
from above and \eqref{eq:almost holo sec4} that
\begin{equation}\label{eq:holo sec1}
\int_{M}||\sigma_i||^2_i\dr_i^n\leq 2C_0 r^{2n}\leq C,
\end{equation}
where $C$ is independent of $i.$ As $\tilde{\tau}_i=F_i(\tau)$ on $\tilde{\phi}_i(U(x;\de_0)),$
and $F_i$ is almost holomorphic by the condition ($C_2$), it follows from \eqref{eq:almost holo sec} that $||\dbar v_i||\leq c\de.$ Then the standard elliptic estimate implies that
\begin{equation}\label{eq:almost holo sec7}
\sup_{\tilde{\phi}_i(U(x;2\de_0)\cap\phi(B_1(o,g_x)))}||v_i||_i^2\leq
C(\de_0 r)^{-2n}\int_{M}||v_i||^2_i\dr_i^n\leq C\de_0^{-2n}\nu.
\end{equation}
As $C$ is uniform, let $\de,\e$ be small enough such that $\nu=\nu(\de,\e)$ satisfies that $20C\nu\leq\de_0^{2n}.$ Then it follows from \eqref{eq:almost holo sec} and \eqref{eq:almost holo sec7} that
\begin{equation}\label{eq:holo sec2}
||\sigma_i||_i\geq||F_i(\tau)||_i-||v_i||_i>e^{-\frac{1}{2}}-\frac{1}{4}>\frac{1}{3}
\quad\textit{on}\;\tilde{\phi}_i(U(x;2\de_0)\cap\phi(B_1(o,g_x))).
\end{equation}
On the other hand by the derivative estimate of holomorphic sections in Lemma \ref{lem-T4.2}
and \eqref{eq:holo sec1} we have
\begin{equation}\label{eq:holo sec3}
\sup_M||\nabla\sigma_i||_i\leq C'l^{\frac{n+1}{2}}\left(\int_{M}||\sigma_i||^2_i\dr_i^n\right)^{\frac{1}{2}}\leq
C''r^{-1}.
\end{equation}
By the choice of $\phi,$ if $\e$ is much smaller than $\de_0$ then for some $u\in\partial B_1(o,g_x),$ we have
$$d_\infty(x,\phi(\de_0 u))\leq d_\infty(x,\phi(\e u))+d_\infty(\phi(\e u),\phi(\de_0 u))\leq 2\de_0 r,$$ where $\phi(\e u)\in\partial U(x;\e).$ Then for sufficiently large $i$
it follows that $$d_i(x_i,\tilde{\phi}_i(\phi(\de_0 u)))\leq 5\de_0 r.$$
Then it follows from \eqref{eq:holo sec2} and \eqref{eq:holo sec3} that
$$||\sigma_i||_i(x_i)\geq\frac{1}{3}-5C''\de_0.$$
Thus we can choose $\de_0\leq\frac{1}{30C''}$ then $||\sigma_i||_i(x_i)\geq\frac{1}{6}$. Combining the fact \eqref{eq:holo sec1}, we can see that \eqref{eq:bergman-local} holds, and consequently,
Theorem \ref{thm-main} is completed.

Also by the partial $C^0$-estimate,
we have the following structure theorem as Theorem 5.9 of \cite{Ti15}:
\begin{theorem}\label{thm-T5.9}
The Gromov-Hausdorff limit $M_\infty$ is a normal variety embedded in some $\CP^{N}$ whose singular set
is a subvariety $\overline{\Sc}$ of complex codimension at least 2. If $T<1$ $\Sc$ is a subvariety consisting
of a divisor $D_\infty$ and a subvariety $\overline{\Sc}$ of complex codimension at least 2. If
$T=1,\Sc=\overline{\Sc}.$ Moreover $D_\infty$ is the limit of $D$ under the Gromov-Hausdorff convergence.
\end{theorem}

\section{Reductivity of the automorphism group of the limit space}

In this section we will follow Lemma 6.9 in \cite{Ti15} to prove Corollary \ref{cor-reductivity}. As $t_i\to T,$ by taking a subsequence we may assume $(M,D,\dr_i)$ converge to $(M_\infty,D_\infty,\dr_\infty).$ By the partial $C^0$-estimate in Theorem \ref{thm-main}, similar to \cite{Ti15}, we have
\begin{enumerate}
\item [(1)]$M$ is embedded in $\CP^{N}$ through an orthonormal basis of $H^0(M,K_M^{-l})$ given by $H_i=H_{\dr_i}.$
\item [(2)]$M_\infty\subset\CP^{N}$ is a normal subvariety with a divisor $D_\infty.$
\item [(3)]There are $\sigma_i\in G=SL(N+1,\C)$ such that $(\sigma_i(M),\sigma_i(D))$ converge to $(M_\infty,D_\infty).$
\end{enumerate}
It follows that the stabilizer $G_\infty$ of $(\sigma_i(M),\sigma_i(D))$ in $G$ contains a nontrivial
holomorphic subgroup. We want to show the Lie algebra $\eta_\infty$ of $G_\infty$ is reductive. For this
target, we also need two steps. First we show that any holomorphic field in $\eta_\infty$ is a complexification of a Killing field on $M_\infty.$ Then we can show that any Killing field can be
extended to be the imaginary part of a holomorphic field on $\CP^{N}.$

As \cite{Ti15}, let $X$ be a holomorphic vector field on $\CP^{N}$ which is tangent to $M_\infty.$ We need
to show that there is a bounded function $\theta_\infty$ such that $i_X\dr_\infty=\sqrt{-1}\dbar\theta_\infty$
on $M_\infty\setminus\overline{\Sc}\cup D_\infty.$ Let $\phi_s$ be an one-parameter subgroup of automorphisms
generated by $Y=Re X$ or $Im X$ then we have
$$\phi_s^*\dr_\infty=\frac{1}{l}\dr_{FS}+\ddb\psi_s.$$ Let $\phi_s$ act on \eqref{eq:path} and note that $\dr_i$ convergent to $\dr_\infty$ smoothly on $M_\infty\setminus\overline{\Sc}\cup D_\infty,$ it follows that
$$Ric(\phi^*_s\dr_\infty)=T\phi^*_s\dr_\infty+(1-T)b_0\al_0$$
on $M_\infty\setminus\overline{\Sc}\cup D_\infty.$ Compare the equations in case of $s$ and 0, and suitably choose $\psi_s$ it follows that
\begin{equation}\label{eq:ma-reg}
\phi^*_s\dr_\infty^n=(\dr_\infty+\ddb\xi_s)^{n}=e^{-T\xi_s}\dr_\infty^n\;\textmd{on}\;M_\infty\setminus
\overline{\Sc}\cup D_\infty,
\end{equation}
where $\xi_s=\psi_s-\psi_0.$ As \cite{Ti15}, first we can show that $\psi_s$ or $\xi_s$ are bounded and continuous functions by the partial $C^0$-estimate. To see the continuity, we note that
\begin{align*}
\dr_\infty&=\frac{1}{l}\dr_{FS}+\ddb\psi_0,\\
\phi^*_s\dr_\infty&=\frac{1}{l}\phi^*_s\dr_{FS}+\ddb\psi_0\circ\phi_s
=\frac{1}{l}\dr_{FS}+\ddb\psi_s.
\end{align*}
Thus we have
$$\xi_s=\psi_s-\psi_0=\psi_0\circ\phi_s-\psi_0+\zeta_s,$$
where $\zeta_s$ is smooth on $\CP^N$ and satisfies
$$\phi^*_s\dr_{FS}=\dr_{FS}+l\ddb\zeta_s,\quad\zeta_0=0.$$
By the partial $C^0$-estimate a subsequence of $\{\log\rho_{i,l}\}$ converges to $l\psi_0+c$ for some constant $c$ as $(M,\dr_i)$ converge to $(M_\infty,\dr_\infty).$
This follows from the the definition of $\rho_{i,l}$ in \eqref{eq:bergman} and the fact
$\ddb\rho_{i,l}=\dr_{FS}-l\dr_i.$ By the gradient estimate in Lemma \ref{lem-T4.2}
$\rho_{i,l}$ are uniformly continuous for each fixed $l.$ Moreover $\rho_{i,l}$
are uniformly bounded by a positive constant, it follows that $\psi_0$ is continuous and thus $\xi_s$ is continuous. We also see that $|\xi_s|\leq\frac{1}{2}$ for sufficiently small $s.$

It follows from \eqref{eq:ma-reg} that
\begin{equation}\label{eq:ma-reg-diff}
-\ddb\xi_s\wedge\left(\sum_{i=0}^{n-1}\dr_\infty^i\wedge
\phi^*_s\dr_\infty^{n-i-1}\right)=(1-e^{-T\xi_s})\dr_\infty^n.
\end{equation}

Recall that as $(M,\dr_i)$ converge to $(M_\infty,\dr_\infty),$ the limit of $D\subset(M,\dr_i)$ converge to a divisor $D_\infty\subset M_\infty$ modulo a singular set $\bar{\Sc}$ with complex codimension at least 2. By section 3 we have a holomorphic section $\tau_\infty$ whose zero set contains $\bar{\Sc}\cup D_\infty.$ In particular
$M_\infty\setminus\tau_\infty^{-1}(0)$ is contained in the regular part of $(M_\infty,\dr_\infty).$ Thus we can choose a cutoff function $\tilde{\eta}:\R\mapsto\R$ satisfying $\tilde{\eta}(t)=1$ for $t\geq 2,
\tilde{\eta}(t)=0$ for $t\leq 1,|\tilde{\eta}'(t)|\leq 1,$ and $|\tilde{\eta}''(t)|\leq 4.$ For any $\e>0,$ we can define $\gamma_\e(x)=\tilde{\eta}(\e\log(-\log||\tau_\infty||_0^2(x))).$

Multiply \eqref{eq:ma-reg-diff} by $\gamma_\e^2\xi_s$ and integrate by parts, we have
\begin{align*}
\int_{M_\infty}\gamma_\e^2\xi_s(1-e^{-T\xi_s})\dr_\infty^n&=\int_{M_\infty}
\sqrt{-1}\partial(\gamma_\e\xi_s)\wedge\dbar\xi_s\wedge
\left(\sum_{i=0}^{n-1}\dr_\infty^i\wedge\phi^*_s\dr_\infty^{n-i-1}\right)\\
&\geq\frac{3}{4}\int_{M_\infty}\sqrt{-1}\gamma_\e^2\partial\xi_s\wedge\dbar\xi_s\wedge
\dr_\infty^{n-1}\\
&-\int_{M_\infty}\sqrt{-1}\xi_s^2\partial\gamma_\e\wedge\dbar\gamma_\e\wedge
\left(\sum_{i=0}^{n-1}\dr_\infty^i\wedge\phi^*_s\dr_\infty^{n-i-1}\right).
\end{align*}
For the last term, recall that $\dr_\infty=\frac{1}{l}\dr_{FS}+\ddb\psi_0$ and
$\phi^*_s\dr_\infty=\frac{1}{l}\dr_{FS}+\ddb\psi_s$ where $\psi_0$ and $\psi_s$
are uniformly bounded functions by the partial $C^0$-estimate. It follows from
similar computations in Lemma \ref{lem-T5.8} combined with standard pluripotential theory (see Lemma 6.10 in \cite{Ti15} for the details) that the last term tends
to 0 as $\e\to 0.$ Thus we have the following inequality if $s$ is so small such that $|\xi_s|\leq\frac{1}{2}$:
\begin{equation}\label{eq:ma-reg-gradient1}
\frac{1}{n}\int_{M_\infty}|\nabla\xi_s|^2\dr_\infty^n\leq 2T\int_{M_\infty}|\xi_s|^2\dr_\infty^n.
\end{equation}
Still use the $\tilde{\eta}$ above we set
$\bar{\gamma}_{\de}(x)=1-\tilde{\eta}(\de^{-1}||\tau_\infty||_0(x)).$ Then
$\bar{\gamma}_{\de}(x)=1$ when $||\tau_\infty||_0(x)\leq\de$ and has its support in the subset $E_\de:=\{x\in M_\infty|||\tau_\infty||_0(x)\leq 2\de\}.$

Recall that $\xi_s=\psi_0\circ\phi_s-\psi_0+\zeta_s$ and note that $\zeta_s$ is defined on $\CP^n$ by $\phi^*_s\dr_{FS}=\dr_{FS}+l\ddb\zeta_s,$ thus $|\zeta_s|\leq C_\zeta s.$
On the other hand, as $\psi_0$ is defined by $\dr_{\infty}=\frac{\dr_{FS}}{l}+l\ddb\psi_0,$ thus $|d\psi_0|\leq C'_\de$ outside $E_{\de/4}$ for some $C'_\de.$ Note that $\phi_0$ is the identity map, for $t$ small it follows that
$$|\xi_s|\leq|\psi_0\circ\phi_s-\psi_0|+|\zeta_s|\leq(C'_\de\sup_{M_\infty\setminus E_{\de/4}}+C_\zeta)s\quad\textmd{on}\;M_\infty\setminus E_{\de/2}.$$ Then we have
\begin{equation}\label{eq:ma-reg-L2}
\int_{M_\infty\setminus E_{\de/2}}|\xi_s|^2\dr_\infty^n\leq C''_\de s^2.
\end{equation}
Combine \eqref{eq:ma-reg-gradient1} and \eqref{eq:ma-reg-L2} it follows that
\begin{align}\label{eq:ma-reg-gradient2}
\int_{M_\infty}|\nabla(\bar{\gamma}_\de\xi_s)|^2\dr_\infty^n&\leq \frac{3}{2}\int_{M_\infty}|\nabla\xi_s|^2\dr_\infty^n+
10\int_{M_\infty}|\nabla\bar{\gamma}_\de|^2|\xi_s|^2\dr_\infty^n\nonumber\\
&\leq 3n\left(\int_{E_{2\de}}|\bar{\gamma}_\de\xi_s|^2\dr_\infty^n+
\int_{M_\infty\setminus E_{\de/2}}|\xi_s|^2\dr_\infty^n\right)+s^2 C'''_\de\nonumber\\
&\leq 3n\int_{M_\infty}|\bar{\gamma}_\de\xi_s|^2\dr_\infty^n+s^2 C_\de.
\end{align}
Next we need the following estimate of the first eigenvalue of $E_\de$ that
$\lambda_1(E_\de)\geq 4n$ for $\de$ sufficiently small where
$$\lambda_1(E):=\left\{\frac{\int_E|\nabla v|^2\dr_\infty^n}{\int_E|v|^2\dr_\infty^n}|0\neq v\in C^1(E\setminus S)\cap L^{\infty}(E), v|_{\partial E}=0\right\}$$ for any open $E\subset M_\infty$ with
nonempty boundary $\partial E\subset M_\infty\setminus\Sc.$ This estimate follows from
the smooth approximation in Theorem \ref{thm-shen}, Croke and P. Li's standard estimates (see the clam in page 47 of \cite{Ti15}). Thus it follows from \eqref{eq:ma-reg-L2}, \eqref{eq:ma-reg-gradient2} and the first eigenvalue estimate above that
\begin{align*}
\int_{M_\infty}|\xi_s|^2\dr_\infty^n&\leq 2\left(\int_{M_\infty}|\bar{\gamma}_\de\xi_s|^2\dr_\infty^n+
\int_{M_\infty}|(1-\bar{\gamma}_\de)\xi_s|^2\dr_\infty^n\right)\\
&\leq \frac{2C_\de}{n}s^2+
2\int_{M_\infty\setminus E_{\de/2}}|\xi_s|^2\dr_\infty^n\leq\bar{C}_\de s^2.
\end{align*}
Combining this with \eqref{eq:ma-reg-gradient1} and dividing by $s^2$ we have
\begin{equation}\label{eq:ma-reg-H2}
\int_{M_\infty}(|s^{-1}\nabla\xi_s|^2+|s^{-1}\xi_s|^2)\dr_\infty^n\leq 2C_\de.
\end{equation}
Now Assume $Y=Re X$ as the generator of $\phi_s.$ Recall that $\xi_s=\psi_0\circ\phi_s-\psi_0+\zeta_s$ and $\psi_0$ is smooth on $M_\infty\setminus\bar{\Sc}\cup D_\infty,$ we can see that $s^{-1}\xi_s$ converge
pointwisely to $u$ on $M_\infty\setminus\bar{\Sc}\cup D_\infty$ as $t\to 0.$ Letting $t\to 0,$ by \eqref{eq:ma-reg-H2} we have that
\begin{equation}\label{eq:ma-reg-dev-L2}
\int_{M_\infty}|u|^2\dr_\infty^n\leq C.
\end{equation}
By differentiating \eqref{eq:ma-reg} it follows that $\int_{M_\infty}u\dr_\infty^n=0.$
For any $q\geq 2$ multiply \eqref{eq:ma-reg-diff} by $\gamma_\e\xi_s|\xi_s|^{q-2}$
and integrate by parts and let $\e\to 0,$ we have
\begin{equation}\label{eq:ma-reg-dev-gradient}
\int_{M_\infty}|\nabla|s^{-1}\xi_s|^{q/2}|^2\dr_\infty^n\leq \frac{12nTq^2}{q-1}\int_{M_\infty}|s^{-1}\xi_s|^{q}\dr_\infty^n.
\end{equation}
By \eqref{eq:ma-reg-dev-L2} and the Sobolev inequality, for any $q\geq 2$ there is a uniform constant $C_q$ satisfying
$$\int_{M_\infty}|s^{-1}\xi_s|^q\dr_\infty^n\leq C_q.$$
Furthermore, give any $\e>0$ it follows from above that
$$\int_{E_\de}|s^{-1}\xi_s|^q\dr_\infty^n\leq Vol(E_\de)^{1/n}
\left(\int_{M_\infty}|s^{-1}\xi_s|^{\frac{qn}{n-1}}\dr_\infty^n\right)
^{\frac{n-1}{n}}\leq Vol(E_\de)^{1/n}(C_{\frac{qn}{n-1}})^{\frac{n-1}{n}}
\leq\frac{\e}{3}$$
for sufficiently small $\de.$ Let $s\to 0$ it follows that
$$\int_{E_\de}|u|^q\dr_\infty^n\leq\frac{\e}{3}.$$
On the other hand as $s^{-1}\xi_s$ converge to $u$
outside $D_\infty\cup\Sc$ for $s$ sufficiently small we have
$$\int_{M_\infty\setminus E_\de}|s^{-1}\xi_s-u|^q\dr_\infty^n\leq\frac{\e}{3}.$$
It follows from the estimates above that $s^{-1}\xi_s$ converge to $u$ in any $L^q$-norm. Then let $s\to 0$ we have a similar inequality to \eqref{eq:ma-reg-dev-gradient} for $u.$ By \eqref{eq:ma-reg-dev-L2} and standard
Moser iteration we can show that $u$ is bounded.

Note that each $\psi_s$ is smooth outside $\overline{\Sc}\cup D_\infty$ and satisfies
$$\frac{1}{l}\dr_{FS}+\ddb\psi_s=\phi^*_s\dr_\infty=\frac{1}{l}\phi^*_s\dr_{FS}+\ddb\phi^*_s\psi_0.$$
It follows that $\psi_s=\phi^*_s\psi_0+\zeta_s$ where $\phi^*_s\dr_{FS}=\dr_{FS}+l\ddb\zeta_s.$ Note that
$\zeta_s$ is a smooth function on $\CP^{N}$ as well as in $s.$ Thus we have
$$u=Y(\psi_0)+\theta_u,\quad \textmd{where}\quad\theta_u=\frac{\partial\zeta_s}{\partial s}|_{s=0}.$$
Similarly by taking $Y$ to be the imaginary part of $X$ we can get a bounded function $v=Y(\psi_0)+\theta_v.$
Now set $\theta_\infty=u+\sqrt{-1}v$ and $\theta=\theta_u+\theta_v,$ then $\theta_\infty=X(\psi_0)+\theta$
is bounded on $M_\infty$ and $i_X\dr_\infty=\sqrt{-1}\dbar\theta_\infty$ since $i_X\dr_{FS}=l\sqrt{-1}\dbar\theta.$ Moreover it holds that
$$\int_{M_\infty}|\nabla\theta_\infty|^2\dr_\infty^n<\infty\quad \textmd{and}\quad\int_{M_\infty}\theta_\infty\dr_\infty^n=0.$$

Next we need to show that $\theta_\infty$ satisfies an eigenfunction equation in a weak sense. Similar
to \cite{Ti15}, by using a test function $\zeta$ and taking the derivative of $s$ to \eqref{eq:ma-reg},
it follows that $$\int_{M_\infty}X(\zeta)\dr_\infty^n=T\int_{M_\infty}\zeta\theta_\infty\dr_\infty^n$$
which implies that in the weak sense,
\begin{equation}\label{eq:eigen}
-\Delta_\infty\theta_\infty=T\theta_\infty\quad on\quad M_\infty.
\end{equation}
By Theorem \ref{thm-shen}, there are smooth \ka\ metrics $\tilde{\dr}_i$ with $Ric(\tilde{\dr}_i)\geq
t_i\tilde{\dr}_i$ and converging to $(M_\infty,\dr_\infty)$ in the Cheeger-Gromov topology, where
$t_i\to T$ as $i\to\infty.$ Similar to \cite{Ti15}, we can show that any bounded function $\theta$
satisfying \eqref{eq:eigen} is the limit of eigenfunctions $\theta_i$ on $M$ which satisfies
$-\Delta_i\theta_i=\lambda_i\theta_i$ with $\lambda_i\to T$ as $t_i\to T,$ where $\Delta_i$ denotes the Laplacian of $\tilde{\dr}_i.$ It follows from Bochner's formula that $\lambda_i\geq t_i,$ and moreover
$$\int_M|\overline{\nabla}\dbar\theta_i|^{2}\tilde{\dr}_i^n=
\int_M((\Delta_i\theta_i)^2-Ric(\partial\theta_i,\partial\theta_i))\tilde{\dr}_i^n
\leq\lambda_i(\lambda_i-t_i)\int_M|\partial\theta_i|^2\tilde{\dr}_i^n.$$
It follows that $\overline{\nabla}\dbar\theta=0$ and $\dbar\theta$ induces a holomorphic
vector field $Z$ outside $\Sc$ of $M_\infty.$ Actually \eqref{eq:eigen} has real solutions which
implies that the imaginary part $Y$ of $Z$ is a Killing field. Thus the Lie algebra $\eta_\infty$
is the complexification of a Lie algebra of Killing fields. Finally by same argument in \cite{Ti15}
$Z$ could be extended to a holomorphic vector field on the ambient $\CP^N,$ which implies that
$\eta_\infty$ is reductive. The proof of Corollary \ref{cor-reductivity} is completed.

\section{\ka-Einstein problem without equivariant actions by holomorphic vector fields}

To show the existence of the \ka-Einstein metric on the Fano manifold with holomorphic vector fields
which do not induce equivariant actions, we can proceed along a continuity path which was first considered
by Yau \cite{Yau} and later developed by a lot of geometers. In the proof of Corollary \ref{cor-YTD} we
choose \eqref{eq:path} as our continuity path:
\begin{equation}\label{eq:path1}
Ric(\dr_t)=t\dr_t+(1-t)(\sum_{r=1}^{m}2\pi b_r[D_r]+b_0\al_0).
\end{equation}
Choose a smooth \ka\ metric $\dr\in c_1(M)$ and suppose $\dr_t=\dr+\ddb\varphi_t,$ we can proceed as following
to transform \eqref{eq:path1} to a family of complex \MA\ equations. As $\dr\in c_1(M),$ there exists a smooth
function $f_1$ such that $$Ric(\dr)=t\dr+(1-t)\dr+\ddb f_1.$$ Next as $2\pi\sum\limits_r b_r[D_r]+b_0\al_0\in c_1(M),$ there exists a smooth function $f_2$ such that
$$\dr=\sum\limits_r b_r\Theta_r+b_0\al_0+\ddb f_2,$$ where $\Theta_r$ represents the curvature of some
Hermitian metric $||\cdot||_r$ defined on the holomorphic line bundle associated to $D_r.$ Combine all the equations above we have
\begin{equation}\label{eq:ma-path}
(\dr+\ddb\varphi_t)^n=\displaystyle\frac{e^{f_1+(1-t)f_2-t\varphi_t}\dr^n}{\prod_r||S_r||_r^{2b_r(1-t)}},
\end{equation}
where $\varphi_t$ satisfies that
$$\displaystyle\int_M\frac{e^{f_1+(1-t)f_2-t\varphi_t}\dr^n}{\prod_r||S_r||_r^{2b_r(1-t)}}=1.$$
Suppose the solvable set $I\subset[0,1]$ is the set of $t$ such that $\eqref{eq:path1}$ or $\eqref{eq:ma-path}$ is solvable. Actually we could choose suitable divisors $D_r,$ smooth $(1,1)$ form $\al_0$ and coefficients $b_r<1$ such that $\eqref{eq:path1}$ is solvable at $t=0,$ see for example \cite{GP}.
Thus $I\neq\emptyset.$ Next, we want to show
\begin{lemma}\label{lem-open}
$I$ is an open set in $[0,1].$
\end{lemma}
\begin{proof}
As the linear theory of conic metrics \cite{Do,JMR} is not so clear in case of simple normal crossing divisors, we may need to some more work to go through this difficulty. In fact we need to show that if
\eqref{eq:ma-path} is solvable at $t=t_0\in[0,1)$ then \eqref{eq:ma-path} is solvable for a small neighborhood around $t_0.$ The main observation is that the Ricci curvature of $\dr_{t_0}$ is strictly greater that $t_0.$
In case that $\dr_t$ is a smooth continuity path that is enough to show the openness \cite{Ti90,Ti97,Ti00}.
In our settings similar to \cite{Do} we need to show a H\"older estimate for the linearization $\mathcal{L}_{t_0}=\Delta_{t_0}+t_0$ of \eqref{eq:ma-path} at $t=t_0,$ where $\Delta_{t_0}$ denotes the Laplacian of $\dr_{t_0}.$ The idea here is to apply the approximation theorem \ref{thm-shen} and establish a uniform H\"older estimate for this family of linearizations.

By the proof of Theorem \ref{thm-shen} in \cite{Sh} (also see \cite{Ti15}), given a solution $\varphi_{t_0}$
to \eqref{eq:ma-path} at $t=t_0,$ for any $\de>0,$ we can solve a family of approximated complex \MA\ equations:
\begin{equation}\label{eq:ma-path-app}
(\dr+\ddb\varphi_{t_0,\de})^n=\displaystyle\frac{e^{f_1+(1-t)f_2-t_0\varphi_{t_0,\de}}\dr^n}
{\prod_r(||S_r||_r^2+\de)^{b_r(1-t)}},
\end{equation}
where $\varphi_{t_0,\de}$ satisfies that
$$\displaystyle\int_M\frac{e^{f_1+(1-t)f_2-t_0\varphi_{t_0,\de}}\dr^n}{\prod_r(||S_r||_r^2+\de)^{b_r(1-t)}}=1.$$
Then $\dr_{t,\de}$ converges to $\dr_t$ in the $C^{\infty}(M\setminus D)$ and Gromov-Hausdorff topology. Moreover, by the H\"older estimate of the complex \MA\ equations with conic singularities in \cite{Ti17,Sh1}
it follows that $||\varphi_{t_0}||_{C^{2,\al}(M)}\leq c_\al$ for some $\al\in(0,1)$ with respect to $\dr_{t_0}.$ Check the details of the proof we can see that those estimates only depend on the local geometry of the model conic metric without anything concerning the linear theory thus by the Gromov-Hausdorff convergence of $\dr_{t_0,\de}$ we can show that $||\varphi_{t_0,\de}||_{C^{2,\al}(M)}\leq c'_\al$ for some $\al\in(0,1)$ with respect to $\dr_{t_0,\de}.$

Note that the linearization of \eqref{eq:ma-path-app} is $\mathcal{L}_{t_0,\de}:=\Delta_{t_0,\de}+t_0$ where $\Delta_{t_0,\de}$ denotes the Laplacian of the approximated metric $\dr_{t_0,\de}=\dr+\ddb\varphi_{t_0,\de}.$
It follows from Theorem \ref{thm-shen} that $Ric(\dr_{t_0,\de})\geq t_0\dr_{t_0,\de}.$ Thus it follows from Bochner's formula that the first eigenvalue $\lambda_1(-\Delta_{t_0,\de})>t_0$ and then $\mathcal{L}_{t_0,\de}: C^{\al}(M)\to C^{2,\al}(M)$ is invertible for some $\al$ less than the H\"older exponent in the $C^{2,\al}$-estimate of $\varphi_{t_0}$ above. We want to claim that for this $\al$ and $\de_0\in(0,1)$ there exists a uniform constant $C_\al>0$ such that for any $\psi\in C^{2,\al}(M)$ which satisfies that $\int_M\psi\dr_{t_0}^n=0$ and $\de\in(0,\de_0)$ it follows that
\begin{equation}\label{eq:holder}
||\psi||_{C^{2,\al}(M)}\leq C_{\al}||\mathcal{L}_{t_0,\de}\psi||_{C^{\al}(M)}
\end{equation}
where the norm is with respect to $\dr_{t_0\de}.$ If this claim is false, we can always find a sequence $\de_k\to 0$ such that there exist a sequence of functions $\psi_k$ satisfying that
\begin{itemize}
\item $\int_M\psi_k\dr_{t_0}^n=0;$
\item $||\psi_k||_{C^{2,\al}(M)}=1$ with respect to $\dr_{t_0,\de_k};$
\item $||\mathcal{L}_{t_0,\de_k}\psi_k||_{C^{\al}(M)}\leq 1/k$ with respect to $\dr_{t_0,\de_k}.$
\end{itemize}
As $||\psi_k||_{C^{2,\al}(M)}=1$ with respect to $\dr_{t_0,\de_k},$ combined with the fact that $\dr_{t_0,\de_k}$ converges to $\dr_{t_0}$ in the Gromov-Hausdorff topology with uniform H\"older norm, we can show that for any smaller $\al'\in(0,\al),$ possibly by passing to a subsequence, $\psi_k$ converges to $\psi$ in the norm of $C^{2,\al'}(M)$ such that $||\psi||_{C^{2,\al}(M)}=1$ with respect to $\dr_{t_0}.$ Moreover it follows from this convergence and $||\mathcal{L}_{t_0,\de_k}\psi||_{C^{\al}(M)}\leq 1/k$ that
$||\mathcal{L}_{t_0}\psi||_{C^{\al}(M)}=0$ thus $(\Delta_{t_0}+t_0)\psi=0.$

Now set $M_\e$ as the subset in $M$ outside a tubular neighborhood of $D=\sum_r D_r$ with scale $\e,$ ignoring $t_0$ in the Laplacian and covariant derivatives below, by standard computations we have
\begin{align*}
0=&-\int_{M_\e}(\Delta+t_0)\psi\cdot\Delta\psi+\int_{\partial M_\e}(\Delta+t_0)\psi*\nabla\psi\\=&\int_{M_\e}\nabla((\Delta+t_0)\psi)\cdot\overline{\nabla}\psi\\
=&\int_{M_\e}(\overline{\nabla}\nabla\nabla\psi\cdot\overline{\nabla}\psi
-Ric(\nabla\psi,\overline{\nabla}\psi)+t_0|\nabla\psi|^2)\\
=&\int_{\partial M_\e}\nabla\nabla\psi*\overline{\nabla}\psi-\int_{M_\e}|\nabla\nabla\psi|^2
-\int_{M_\e}(Ric(\nabla\psi,\overline{\nabla}\psi)-t_0|\nabla\psi|^2).
\end{align*}
With respect to $\dr_{t_0},$ since $\ddb\psi$ is $C^{\al},$ along $M_\e$ the third order derivative $\overline{\nabla}\nabla\nabla\psi$ has the order at most $\e^{\al-1},$ so does $\nabla\nabla\psi.$
On the other hand, the measure of $\partial M_\e$ has the order of $\e$ and $\nabla\psi$ is bounded, thus
the first integration over $\partial M_\e$ on the RHS tends to 0 as $\e$ tends to 0. As $Ric(\dr_{t_0})$
is strictly greater than $t_0,$ it follows that $\psi$ must be 0, which contradicts with $||\psi||_{C^{2,\al}(M)}=1.$ Thus the claim \eqref{eq:holder} follows.

By this claim, as $\dr_{t_0,\de}$ converges to $\dr_{t_0}$ in the Gromov-Hausdorff topology with uniform
$C^{2,\al}$ norm, \eqref{eq:holder} is also true for $\mathcal{L}_{t_0}=\Delta_{t_0}+t_0.$ Thus the openness of $I$ follows from the inverse function theorem and this lemma is true.
\end{proof}
\begin{remark}
In fact, if there is no $\al_0$ in \eqref{eq:path1}, we could show the openness result in case that there is no holomorphic vector fields tangential to divisors, which generalizes Donaldson's openness theorem \cite{Do} to the case of simple normal crossing divisors without use of linear theory.
\end{remark}

It remains to show the closeness of the solvable set $I.$ Suppose we have a sequence $t_i\in I$ such that $t_i\to T$ but $t\notin I,$ then by the argument in \cite{JMR,Ti15}, $||\varphi_{t_i}||_{C^0}$ must diverge to $\infty.$ By Corollary \ref{cor-reductivity} and geometrical invariant theory there exists a $C^*$-subgroup
$G_0\subset G$ that degenerates $(M,D)$ to $(M_\infty,D_\infty).$ As $||\varphi_{t_i}||_{C^0}$ diverges to $\infty,$ by \cite{Ti97} the central fiber $(M_\infty,D_\infty)$ of this degeneration is not biholomorphic to
$(M,D).$ We will derive a contradiction with $K$-stability.

Now let us recall the $K$-stability defined by Tian in \cite{Ti97,Ti15}. First recall that
on the Fano manifold $(M,\dr)$ where $\dr\in c_1(M)$ is a smooth \ka\ metric. For any holomorphic
vector field $X$ on $M,$ the Futaki invariant is defined by
\begin{equation}\label{eq:futaki}
f_{M}(X):=-n\int_M\theta_X(Ric(\dr)-\dr)\wedge\dr^{n-1},
\end{equation}
where $i_X\dr=\sqrt{-1}\dbar\theta_X$ and this is a holomorphic invariant by Futaki \cite{Fu}.
Moreover, in \cite{DT} the Futaki invariant was extended to normal Fano varieties as following: Assume
$M\mapsto\CP^N$ through a basis of $H^0(M,K_M^{-l})$ for sufficiently large $l.$ Then for any
algebraic subgroup $G_0=\{\sigma(t)\}_{t\in\C^*}$ of $G=SL(N+1,\C)$ there is a unique limiting cycle
$M_0=\lim\limits_{t\to 0}\sigma(t)(M)\subset\CP^N.$ Let $X$ be the holomorphic vector field whose real part
generates the action by $\sigma(e^{-s}).$ If $M_0$ is normal then \eqref{eq:futaki} could be
generalized to $M_0$ by \cite{DT}. Then we can define the $K$-stability as following:
\begin{definition}\label{def-K}
We say that $M$ is $K$-stable with respect to $K_M^{-l}$ if $Re(f_{M_0}(X))\geq 0$ for any
$C^*$ subgroup $G_0\subset SL(N+1,\C)$ with a normal $M_0,$ and the equality holds if and only if
$M_0$ is biholomorphic to $M.$ We way that $M$ is $K$-stable if it is $K$-stable for all
sufficiently large $l.$
\end{definition}

We also need the Mabuchi energy and the twisted Mabuchi energy in the proof. Given a Fano manifold
$(M,\dr)$ as before, suppose $\{\varphi_s\}$ is a path in $Psh(M,\dr)$ connecting 0 and $\varphi,$ then the Mabuchi energy can be defined as following:
\begin{equation}\label{eq:mabuchi}
M_{\dr}(\varphi):=-n\int_0^1\int_M\dot{\varphi}_s(Ric(\dr_{\varphi_s})-\dr_{\varphi_s})
\wedge\dr_{\varphi_s}^{n-1}ds,
\end{equation}
where $\dr_{\varphi_s}:=\dr+\ddb\varphi_s,$ and moreover this Mabuchi energy is independent of
the choices of paths. Given the $C^*$-action $\{\sigma_s:=\sigma(e^{-s})\}$ above and
suppose $$\frac{1}{l}\sigma_s^*\dr_{FS}=\dr+\ddb\psi_{s},$$ it follows from \cite{DT} that
\begin{equation}\label{eq:mabuchi-dev}
\lim\limits_{s\to+\infty}\frac{d}{ds}M_\dr(\psi_s)=Re(f_{M_0}(X)),
\end{equation}
where $X$ is the generator of $\{\sigma_s\}.$

Similarly we could define the twisted Mabuchi energy as \cite{LS,Sh}:
\begin{align}\label{eq:twisted mabuchi}
M_{\dr,t}(\varphi):=&-n\int_0^1\int_M\dot{\varphi}_s(Ric(\dr_{\varphi_s})-
t\dr_{\varphi_s}-(1-t)(2\pi\sum_r b_r[D_r]+b_0\al_0))
\wedge\dr_{\varphi_s}^{n-1}ds\nonumber\\
=&\int_M\log\frac{\dr_\varphi^n}{\dr^n}\dr_\varphi^n-\int_M h(\dr_\varphi^n-\dr^n)
-\sum_{j=1}^n\int_M\varphi\dr^j\wedge\dr_\varphi^{n-j}\nonumber\\
&+\frac{nt}{n+1}\sum_{j=0}^n\int_M\varphi\dr^j\wedge\dr_\varphi^{n-j}\nonumber\\
&+(1-t)\sum_{j=0}^{n-1}\left(\sum_r b_r\int_{D_r}\varphi\dr^j\wedge\dr_\varphi^{n-1-j}+
b_0\int_M\varphi\al_0\wedge\dr^j\wedge\dr_\varphi^{n-1-j}\right),
\end{align}
which is also independent of paths connecting 0 and $\varphi.$ Similar to the argument in \cite{DT} we can also define the corresponding generalized Futaki invariant on the limit $(M_\infty,D_\infty):$
$$f_{M_\infty,D_\infty,t}(X)=-n\int_{M_\infty}\theta_X(Ric(\dr)-t\dr-(1-t)(2\pi\sum_r b_r[D_r]+b_0\al_0)))\wedge\dr^{n-1},$$
and moreover similar to \eqref{eq:mabuchi-dev}, if the $C^*$-action $\sigma_s$ preserves $(M_\infty,D_\infty),$
it follows that
\begin{equation}\label{eq:general-mabuchi-dev}
\lim\limits_{s\to+\infty}\frac{d}{ds}M_{\dr,t}(\psi_s)=Re(f_{M_\infty,D_\infty,t}(X)),
\end{equation}
To complete the proof, we need the following lemma which is modified from \cite{DT}:
\begin{lemma}\label{lem-DT}
If \eqref{eq:path1} or \eqref{eq:ma-path} is solvable at $t,$ then $Re(f_{M_\infty,D_\infty,t}(X))\geq 0.$ Futhermore, if $\sigma_s$ preserve $(M,D),$ then $Re(f_{M_\infty,D_\infty,t}(X))=0.$
\end{lemma}
\begin{proof}
Let us briefly describe the main ingredients in the proof and for more details, see \cite{DT,LS}.
To show the first conclusion, we first need to show that the corresponding twisted Mabuchi energy
$M_{\dr,t}$ is bounded from below. For this target, we could argue as \cite{LS,Sh} that the twisted Ding energy modified from \cite{Ding} has the same behavior with the twisted Mabuchi energy. Then as \eqref{eq:ma-path} is solvable at $t,$ we could use the log $\al$-invariant argument to show that the twisted Ding energy is bounded from below which implies that the twisted Mabuchi energy is also bounded from below. Then it follows from \eqref{eq:general-mabuchi-dev} that $Re(f_{M_\infty,D_\infty,t}(X))\geq 0$ otherwise $M_{\dr,t}(\psi_s)$ may decrease to $-\infty$ which is impossible. For the second conclusion as $\sigma_s$ preserve $(M,D),$ we can choose $-X$ instead of $X$ and then $Re(f_{M_\infty,D_\infty,t}(X))$ must be 0.
\end{proof}
As \cite{Ti15}, by \eqref{eq:mabuchi} and \eqref{eq:twisted mabuchi}, we have the following relation:
$$TM_{\dr}(\psi_s)=M_{\dr,T}(\psi_s)-(1-T)M_{\dr,0}(\psi_s).$$
Take the derivative with respect to $s$ and let $s\to+\infty,$ it follows from \eqref{eq:mabuchi-dev} and \eqref{eq:general-mabuchi-dev} that
$$TRe(f_{M_\infty}(X))=Re(f_{M_\infty,D_\infty,T}(X))-(1-T)Re(f_{M_\infty,D_\infty,0}(X)).$$
By Lemma \ref{lem-DT}, as the corresponding action $\sigma_s$ is generated from the limit
of $(M,D,\dr_{t_i})$ with $t_i\to T$ it follows that $Re(f_{M_\infty,D_\infty,T}(X))=0.$ On the other hand we have known that in our case \eqref{eq:ma-path} is always solvable at $t=0$ thus
$Re(f_{M_\infty,D_\infty,0}(X))\leq 0.$ Thus it follows that $Re(f_{M_\infty}(X))\leq 0$ as $T>0$
and recall that $M_\infty$ is not biholomorphic to $M$ as $\varphi_{t_i}||_{C^0}$ diverges to infinity, which contradicts with the $K$-stability that $Re(f_{M_\infty,D_\infty,0}(X))\geq 0$
and the equality holds only if $M_\infty$ is biholomorphic to $M.$ Thus we complete the proof of
Corollary \ref{cor-YTD}.
\begin{remark}
If we replace the $K$-stability by the $CM$-stability in the Corollary \ref{cor-YTD} as \cite{Ti15} we can also establish the existence of \ka-Einstein metrics by quite similar arguments in \cite{Ti15}.
\end{remark}

\end{document}